\documentclass[a4paper,11pt]{article}
\usepackage{amssymb,amsmath,latexsym,amsthm}
\usepackage{color,mathrsfs}
\usepackage{url}
\usepackage{enumerate}

\usepackage{dsfont}

\usepackage{endnotes}

\usepackage[dvips]{graphics}
\usepackage[xetex]{graphicx}
\usepackage{fancyhdr}
\usepackage{pstricks,pst-plot,pst-node,pstricks-add}

\title{Renormalized Energy and Asymptotic Expansion of Optimal Logarithmic Energy on the Sphere}
\author{Laurent B\'etermin\\Etienne Sandier}

\newtheorem{thm}{Theorem}[section]
\newtheorem{defi}[thm]{Definition}
\newtheorem{conjecture}[thm]{Conjecture}

\newtheorem{prop}[thm]{Proposition}
\newtheorem{lemma}[thm]{Lemma}
\newtheorem{remark}[thm]{Remark}

\newcommand{\divergence}{\mathop{\rm div}\nolimits}
\newcommand{\curl}{\mathop{\rm curl}\nolimits}
\newcommand{\supp}{\mathop{\rm Supp}\nolimits}

\newcommand{\R}{\mathbb R}
\newcommand{\Q}{\mathbb Q}
\newcommand{\Z}{\mathbb Z}

\newcommand{\C}{\mathbb C}
\newcommand{\ms}{{\mathcal S}}
\newcommand{\V}{{\mathcal V}}
\def\S{{\mathbb S}}
\def\bm{{\overline m}}
\def\mb{{\underline m}}
\def\indic{\mathbf{1}}
\def\ds{\displaystyle}
\def\hal{{\frac12}}
\def\ep{{\varepsilon}}
\def\vp{\varphi}
\def\div{\divergence}

\numberwithin{equation}{section}
\def\Xint#1{\mathchoice
  {\XXint\displaystyle\textstyle{#1}}%
  {\XXint\textstyle\scriptstyle{#1}}%
  {\XXint\scriptstyle\scriptscriptstyle{#1}}%
  {\XXint\scriptscriptstyle\scriptscriptstyle{#1}}%
  \!\int}
\def\XXint#1#2#3{{\setbox0=\hbox{$#1{#2#3}{\int}$}
    \vcenter{\hbox{$#2#3$}}\kern-.5\wd0}}
\def\dashint{\Xint-}

\allowdisplaybreaks

\begin{document}
\maketitle
\begin{abstract}
We study the Hamiltonian of a two-dimensional log-gas  with a confining potential $V$ satisfying the weak growth assumption -- $V$ is of the same order than $2\log\|x\|$ near infinity -- considered by Hardy and Kuijlaars [J. Approx. Theory, 170(0):44-58, 2013]. We prove an asymptotic expansion, as the number $n$ of points goes to infinity, for the minimum of this Hamiltonian using the Gamma-Convergence method of Sandier and Serfaty \cite{2DSandier}. We show that the asymptotic expansion as $n\to +\infty$ of the minimal logarithmic energy of $n$ points on the unit sphere in $\R^3$ has a term of order $n$ thus proving a long standing conjecture of Rakhmanov, Saff and Zhou [Math. Res. Letters, 1:647-662, 1994]. Finally we prove the equivalence between the conjecture of Brauchart, Hardin and Saff [Contemp. Math., 578:31-61,2012] about the value of this term and the conjecture of Sandier and 
Serfaty [Comm. Math. Phys., 313(3):635-743, 2012] about the minimality of the triangular lattice for a ``renormalized energy" $W$ among configurations of fixed asymptotic density. 
\end{abstract}

\noindent
\textbf{AMS Classification:} Primary 52A40, 82B05 ; Secondary 41A60, 82B21, 31C20. \\
\textbf{Keywords:} Coulomb gas ; Abrikosov lattices ; Triangular lattice ; Renormalized energy ; Crystallization ; Logarithmic energy ; Number theory ; Logarithmic potential theory ; Weak confinement ; Gamma-convergence ; Ginzburg-Landau ; Vortices.\\

\section{Introduction}
Let $(x_1,...,x_n)\in(\R^2)^n$ be a configuration of $n$ points interacting through a  logarithmic potential and confined by an external field $V$. The Hamiltonian of this system, also known as a  Coulomb gas,  is defined as
$$
w_n(x_1,...,x_n):=-\sum_{i\neq j}^n\log|x_i-x_j|+n\sum_{i=1}^nV(x_i)
$$
where $|\cdot|$ is the Euclidean norm in $\R^2$. The minimization of  $w_n$ is linked to the following classical problem of logarithmic potential theory: find a probability measure $\mu_V$ on $\R^2$ which minimizes
\begin{equation}\label{IV}
I_V(\mu):=\iint_{\R^2\times \R^2}\left(\frac{V(x)}{2}+\frac{V(y)}{2}-\log|x-y|\right)d\mu(x)d\mu(y)
\end{equation}
amongst all probability measures $\mu$ on $\R^2$. This type of  problem dates back to Gauss. More recent references are the thesis of  Frostman \cite{Frost} and the monography of  E.Saff and V.Totik \cite{SaffTotik}. The usual assumptions on $V:\R^2\to\R\cup\{+\infty\}$ are that it is lower semicontinuous, that it is finite on a set of nonzero capacity, and that it satisfies the growth assumption
\begin{equation}\label{classicgrowth}
\lim_{|x|\to +\infty}\left\{V(x)-2\log| x|\right\}=+\infty.
\end{equation}
These assumptions ensure that a unique minimizer $\mu_V$ of $I_V$ exists and that it has compact support. 

Recently, Hardy and Kuijlaars \cite{Hardy2} (see also \cite{Hardy}) proved that if one replaces \eqref{classicgrowth} by  the so-called  weak growth assumption 
\begin{equation}\label{growthL}
\liminf_{|x|\to +\infty}\left\{V(x)-\log(1+|x|^2)  \right\}>-\infty,
\end{equation}
then $I_V$ still admits a unique minimizer, which may no longer have compact support. Moreover Bloom, Levenberg and Wielonsky \cite{Bloom:2014qy} proved that the classical Frostman type inequalities still hold in this case. These results make use of the stereographic projection, a method already used by Rakhmanov, Saff and Zhou in \cite{Electsphere} to prove separation properties of optimal configurations on spheres.\\

Coming back to the minimum of the discrete energy $w_n$, its relation to the minimum of $I_V$ is that as $n\to +\infty$, the minimum of $w_n$ is equivalent to $n^2 \min I_V$. The next term in the asymptotic expansion of $w_n$ was derived by  Sandier and Serfaty \cite{2DSandier}  in the classical case \eqref{classicgrowth},  it reads 
$$\min w_n =  n^2 \min I_V - \frac{n}{2} \log n + \alpha_V n +o(n),$$
where $\alpha_V$ is related to the minimum of a Coulombian renormalized energy studied in \cite{Sandier_Serfaty} which quantifies the discrete energy of infinitely many positive charges  in the plane  screened by a uniform negative background. Note that rather strict assumptions in addition to \eqref{classicgrowth} need to be made on $V$ for this expansion to hold, but they are satisfied in particular if $V$ is smooth and  strictly convex.

Here, we show that such an asymptotic formula still  holds when the classical growth assumption \eqref{classicgrowth} is replaced with  the weak growth assumption \eqref{growthL}. However it is no longer obvious that the minimum of $w_n$ is achieved in this case, as the weak growth assumption could allow one point to go to infinity.
\begin{thm} ~\label{thintro}
 Let $V$ be an admissible potential \footnote{See Section \ref{ad} for the precise definition}. Then the following asymptotic expansion holds.
\begin{equation}\label{expansion}
\inf_{(\R^2)^n} w_n=I_V(\mu_V)n^2-\frac{n}{2}\log n +\left(\frac{1}{\pi}\min_{\mathcal{A}_1}W-\frac{1}{2}\int_{\R^2}m_V(x)\log m_V(x)dx\right)n + o(n),
\end{equation}
where $\mu_V = m_V(x)\,dx$ is the unique minimizer of $I_V$  (see Section~\ref{RE} for precise definitions of $W$ and $\mathcal{A}_1$.)
\end{thm}

This result is proved  using the methods in \cite{Sandier_Serfaty,2DSandier} suitably adapted to equilibrium measures with possibly non-compact support together using  the  stereographic projection  as in \cite{Electsphere}, or more recently in \cite{Dragnev:2002, Hardy2,Hardy,Bloom:2014qy}, which allows also  to connect the discrete energy problem for log gases in the plane with the discrete logarithmic energy problem for finitely many points on the unit sphere $\S^2$ in the Euclidean space $\R^3$.\\

The logarithmic energy of a configuration $(y_1,...,y_n)\in (\S^2)^n$ is given by
$$
\textnormal{E}_{\log}(y_1,...,y_n):=-\sum_{i\neq j}^n\log\|y_i-y_j\|,
$$
where $\|\cdot\|$ is the Euclidean norm in $\R^3$. Finding a minimizer of such an energy functional is a problem with many links and ramifications  as discussed in the fundamental paper of Saff and Kuijlaars \cite{PointsSphere} (see also \cite{BrauchGrab}). For instance  Smale's $7^{th}$ problem \cite{231301} is  to find, for any $n\geq 2$, a universal constant $c\in\R$ and a nearly optimal configuration $(y_1,...,y_n)\in (\S^2)^n$ such that, letting  $\mathcal{E}_{\log}(n)$ denote the minimum of $\textnormal{E}_{\log}$ on $(\S^2)^n$, 
$$
\textnormal{E}_{\log}(y_1,...,y_n)-\mathcal{E}_{\log}(n)\leq c\log n.
$$
Identifying the term of order $n$  in the expansion of $\mathcal{E}_{\log}(n)$ can be seen as a modest step towards  a better understanding  of this problem.\\

It was known  (lower bound by Wagner \cite{57451281} and upper bound by Kuijlaars and Saff \cite{403978}), that 
$$
\left(\frac{1}{2}-\log 2\right)n^2-\frac{1}{2}n\log n+c_1 n\leq \mathcal{E}_{\log}(n)\leq \left(\frac{1}{2}-\log 2\right)n^2-\frac{1}{2}n\log n+c_2 n
$$
for some fixed constant $c_1$ and $c_2$. Thus one can naturally ask for the existence of the limit
 $$
 \lim_{n\to +\infty}\frac{1}{n}\left[\mathcal{E}_{\log}(n)-\left(\frac{1}{2}-\log2\right)n^2+\frac{n}{2}\log n  \right].
 $$
\medskip

\begin{conjecture} \label{conj1}(Rakhmanov, Saff and Zhou, \cite{2400418}) There exists a constant $C$ not depending on $n$ such that
$$
\mathcal{E}_{\log}(n)=\left(\frac{1}{2}-\log2  \right)n^2-\frac{n}{2}\log n+ Cn+o(n) \quad \text{ as } n\to +\infty.
$$
\end{conjecture}
\medskip 

\begin{conjecture} \label{conj2} (Brauchart, Hardin and Saff, \cite{Brauchart}) The constant $C$ in Conjecture \ref{conj1} is equal to $C_{BHS}$, where 
\begin{equation}\label{CBHSdef}
C_{BHS}:=2\log2 +\frac{1}{2}\log\frac{2}{3}+3\log\frac{\sqrt{\pi}}{\Gamma(1/3)}.
\end{equation}
\end{conjecture}
\medskip

As we will see, our results imply that the last conjecture is  equivalent to one  concerning  the global optimizer of the renormalized energy $W$. \medskip

\begin{conjecture} \label{conj3} (Sandier and Serfaty, \cite{Sandier_Serfaty}, or see the review by Serfaty \cite{Serfaty:2013fk}) The triangular lattice is a  global minimizer of $W$  among discrete subsets of $\R^2$ with asymptotic density one.
\end{conjecture}
\medskip

The expansion \eqref{expansion} in the particular case $V(x)=\log(1+|x|^2)$ transported to  $\S^2$ using an inverse stereographic projection and appropriate rescaling gives an expansion for $\mathcal{E}_{\log}(n)$ and thus proves Conjecture \ref{conj1}. The constant $C$ in Conjecture \ref{conj1} can moreover be  expressed in terms of the minimum of the renormalized energy $W$. The  value of $W$ for the triangular lattice obviously provides an upper bound for this minimum, and  by using the Chowla-Selberg formula  to compute the expression given in \cite{Sandier_Serfaty} for this quantity, we show that this upper bound is precisely $C_{BHS}$. This bound is of course sharp if and only if Conjecture \ref{conj3} is true. Thus we deduce from \eqref{expansion} the following.

\begin{thm}  There exists $C\neq 0$ independent of $n$ such that, as $n\to +\infty$, 
$$
\mathcal{E}_{\log}(n)=\left(\frac{1}{2}-\log2  \right)n^2-\frac{n}{2}\log n+Cn+o(n),\quad  C= \frac{1}{\pi}\min_{\mathcal{A}_1}W +\frac{\log\pi}{2}+\log2.$$
Moreover  $\displaystyle C\leq C_{BHS}$ where $C_{BHS}$ is given in \eqref{CBHSdef}, and equality holds iff $\displaystyle \min_{\mathcal{A}_1}W$ is achieved for the triangular lattice of density one.
\end{thm}

The plan of the paper is  as follows. In Section \ref{RE} we recall the definition of $W$ and some of its properties from \cite{Sandier_Serfaty}. In Section \ref{EPWP} we recall results about existence, uniqueness and variational Frostman inequalities for $\mu_V$. Moreover, we give the precise definition of an admissible potential $V$. In Sections  \ref{SpF} and  \ref{UsLem}  we adapt the method of \cite{Sandier_Serfaty} to the case of equilibrium measures with noncompact support. The expansion \eqref{expansion} is proved in Section \ref{AEOTH}. Finally in Section \ref{CTLEOTS} we prove Conjecture \ref{conj1} about the existence of $C$, the upper bound  $C\leq C_{BHS}$ and the equivalence between Conjectures \ref{conj2} and \ref{conj3}.

\section{Renormalized Energy}\label{RE}
Here we recall the definition of the renormalized energy $W$ (see \cite{2DSandier} for more details). For any $R>0$, $B_R$ denotes the ball centered at the origin with radius $R$.

\begin{defi} Let $m$ be a nonnegative number and $E$ be a vector-field in $\R^2$. We say $E$ belongs to the \textbf{admissible class $\mathcal{A}_m$} if 
\begin{equation}\label{div} \divergence E=2\pi(\nu-m) \text{ and } \curl E =0
\end{equation}
where $\nu$ has the form
\begin{equation}\label{nu}
\nu=\sum_{p\in \Lambda}\delta_p \text{ , for some discrete set } \Lambda\subset\R^2,
\end{equation}
and if 
\begin{equation*}
\frac{\nu(B_R)}{|B_R|} \text{ is bounded by a constant independent of } R>1.
\end{equation*}
\end{defi}
\begin{remark} The real $m$ is the average density of the points of $\Lambda$ when $E\in\mathcal{A}_m$.
\end{remark}

\begin{defi}
Let $m$ be a nonnegative number. For any continuous function $\chi$ and any vector-field $E$ in $\R^2$ satisfying \eqref{div} where $\nu$ has the form \eqref{nu} we let
\begin{equation*}
W(E,\chi)=\lim_{\eta\to 0}\left(\frac{1}{2}\int_{\R^2\backslash\cup_{p\in\Lambda}B(p,\eta)}\chi(x) |E(x)|^2dx+\pi\log\eta\sum_{p\in \Lambda}\chi(p)   \right).
\end{equation*}
\end{defi}

\begin{flushleft} We use the notation $\chi_{B_R}$ for positive cutoff functions satisfying, for some constant $C$ independent of $R$
\end{flushleft}
\begin{equation} \label{chi}
|\nabla\chi_{B_R}|\leq C,\quad \supp(\chi_{B_R})\subset B_R,\quad \chi_{B_R}(x)=1 \text{ if } d(x,B_R^c)\geq 1.
\end{equation}
where $d(x,A)$ is the Euclidean distance between $x$ and set $A$.
\begin{defi} The \textbf{renormalized energy $W$} is defined, for $E\in \mathcal{A}_m$ and $\{\chi_{B_R}\}_R$ satisfying \eqref{chi}, by
\begin{equation*}
W(E)=\limsup_{R\to +\infty}\frac{W(E,\chi_{B_R})}{|B_R|}.
\end{equation*}
\end{defi}

\begin{remark} It is shown in \cite[Theorem 1]{Sandier_Serfaty} that the value of $W$ does not depend on the choice of  cutoff functions satisfying \eqref{chi}, and that 
  $W$ is bounded below and {\em admits a minimizer over $\mathcal{A}_1$}.
  
  Moreover (see \cite[Eq. (1.9),(1.12)]{Sandier_Serfaty}),  if $E\in\mathcal{A}_m$, $m>0$, then 
$$
E'=\frac{1}{\sqrt{m}}E(./\sqrt{m})\in\mathcal{A}_1\quad \text{and} \quad W(E)=m\left(W(E')-\frac{\pi}{2}\log m  \right).
$$ In particular 
\begin{equation}\label{scaling}
\min_{\mathcal{A}_m} W=m\left( \min_{\mathcal{A}_1} W- \frac{\pi}{2}\log m  \right),
\end{equation}
and  $E$ is a minimizer of $W$ over $\mathcal{A}_m$ if and only if $E'$ minimizes $W$ over $\mathcal{A}_1$.
\end{remark}

In the periodic case, the following result \cite[Theorem 2]{Sandier_Serfaty}, which supports Conjecture~\ref{conj3} above: Given a  Bravais lattice\footnote{A Bravais lattice of $\R^2$, also called ``simple lattice" is $L=\Z \vec{u}\oplus \Z \vec{v}$ where $(\vec{u},\vec{v})$ is a basis of $\R^2$.} $\Lambda$ of density $m$, there is a unique (modulo constants)  $\Lambda$-periodic solution $H_\Lambda$ to the equation  $-\Delta H = 2\pi(\sum_{p\in\Lambda} \delta_p - m)$  and we may define
$$ W(\Lambda) = W(\nabla H_\Lambda).$$
Then we have 
\begin{thm}\label{Trioptimal}
The unique minimizer, up to rotation, of $W$ over Bravais lattices of fixed density $m$ is the triangular lattice 
$$
\Lambda_m=\sqrt{\frac{2}{m\sqrt{3}}}\left(\Z (1,0)\oplus \Z \left(\frac{1}{2},\frac{\sqrt{3}}{2}\right)\right).
$$
\end{thm} 
This is proved in \cite{Sandier_Serfaty} using the result of Montgomery on minimal theta function \cite{Mont}, we provide an alternative proof below. 
\begin{proof}
Osgood, Phillips and Sarnak \cite[Section 4, page 205]{OPS} proved, for $\Lambda=\Z\oplus \tau\Z$, $\tau=a+ib$, that the height of the flat torus $\C/\Lambda$ (see \cite{OPS,Chiu,CoulLazzarini} for more details) is 
$$
h(\Lambda)=-\log(b|\eta(\tau)|^4)+C, \quad C\in \R,
$$ where $\eta$ is the Dedekind eta function\footnote{See Section \ref{NTh}}. But  from \cite{Sandier_Serfaty} we have 
$$
W(\Lambda)=-\frac{1}{2}\log\left(\sqrt{2\pi b}|\eta(\tau)|^2\right)+C,
$$ 
therefore $W(\Lambda)=\alpha h(\Lambda)+\beta$ 
where $\alpha>0$, $\beta\in\R$ are independent of $\Lambda$.

Then from  \cite[Corollary 1(b)]{OPS}, the triangular lattice minimizes $h$ among Bravais lattices with fixed density, hence the same is true  for $W$.
\end{proof}

\section{Equilibrium Problem in the Whole Plane}\label{EPWP}
In this section we recall results on existence, uniqueness and characterization of the equilibrium measure $\mu_V$ and we give the definition of the admissible potentials.

\subsection{Equilibrium measure, Frostman inequalities and differentiation of $U^{\mu_V}$}

\begin{defi} \label{defUmu}\textnormal{(\cite{Bloom:2014qy})} Let $K\subset \R^2$ be a compact set and let $\mathcal{M}_1(K)$ be the family of probability measures supported on $K$. Then the \textbf{logarithmic potential} and the \textbf{logarithmic energy} of $\mu\in \mathcal{M}_1(K)$ are defined as 
$$
U^{\mu}(x):=-\int_{K}\log|x-y|d\mu(y) \quad \text{and} \quad I_0(\mu):=-\iint_{K\times K} \log|x-y|d\mu(x)d\mu(y).
 $$
We say that $K$ is \textbf{log-polar} if $I_0(\mu)=+\infty$ for any $\mu\in \mathcal{M}_1(K)$ and we say that a Borel set $E$ is log-polar if every compact subset of $E$ is log-polar.
Moreover, we say that an assertion holds \textbf{quasi-everywhere} (q.e.) on $A\subset \R^2$ if it holds on $A\backslash P$ where $P$ is log-polar. 
\end{defi}
\begin{remark} We recall that the Lebesgue measure of  a log-polar set is zero.
\end{remark}
\noindent Now we recall results about the existence, the uniqueness and the characterization of the equilibrium measure $\mu_V$ proved in \cite{Frost,SaffTotik} for the classical growth assumption \eqref{classicgrowth}, and by Hardy and Kuijlaars \cite{Hardy,Hardy2} (for existence and uniqueness) and Bloom, Levenberg and Wielonsky \cite{Bloom:2014qy} (for Frostman type variational inequalities) for weak growth assumption \eqref{growthL}.

\begin{thm}\label{equi}\textnormal{(\cite{Frost,SaffTotik,Hardy,Hardy2,Bloom:2014qy})} Let $V$ be a lower semicontinuous function on $\R^2$ such that $\{x\in \R^2 ; V(x)<+\infty \}$ is a non log-polar subset of $\R^2$ satisfying 
$$
\liminf_{|x|\to +\infty}\{V(x)-\log(1+|x|^2) \}>-\infty.
$$
Then we have:
\begin{enumerate}
\item $\displaystyle \inf_{\mu \in \mathcal{M}_1(\R^2)}I_V(\mu)$ is finite, where $I_V$ is given by \eqref{IV}.
\item There exists a unique equilibrium measure $\mu_V\in \mathcal{M}_1(\R^2)$ with 
$$
 I_V(\mu_V)=\inf_{\mu \in \mathcal{M}_1(\R^2)}I_V(\mu)
$$
and the logarithmic energy $I_0(\mu_V)$ is finite.
\item The support $\Sigma_V$ of $\mu_V$ is contained in $\{x\in \R^2 ; V(x)<+\infty\}$ and $\Sigma_V$ is not log-polar.
\item Let 
\begin{equation}\label{cvdef}
 c_V:=I_V(\mu_V)-\int_{\R^2}\frac{V(x)}{2}d\mu_V(x)
 \end{equation}
denote the Robin constant. Then we have the following Frostman variational inequalities (for the fact that $U^{\mu_V}$ is well defined, see for instance \cite{Bloom:2014qy}):
\begin{align} \label{Frost1}& U^{\mu_V}(x)+\frac{V(x)}{2}\geq c_V \quad \text{q.e. on } \R^2,\\
\label{Frost2} & U^{\mu_V}(x)+\frac{V(x)}{2}\leq c_V \quad \text{for all } x\in\Sigma_V.
\end{align} 
\end{enumerate}
\end{thm}
\begin{remark}
In particular we have $\displaystyle U^{\mu_V}(x)+\frac{V(x)}{2}=c_V$ q.e. on $\Sigma_V$.
\end{remark}

 As  in \cite{Electsphere}, \cite{Dragnev:2002}, or more recently in  \cite{Hardy}, the hypothesis of Theorem~\ref{equi} are usefully transported to the sphere  $\mathcal{S}$ in $\R^3$ centred at $(0,0,1/2)$ with radius $1/2$,  by the inverse stereographic projection $T:\R^2\to \mathcal{S}$  defined by
\begin{equation*}
T(x_1,x_2)=\left(\frac{x_1}{1+|x|^2}, \frac{x_2}{1+|x|^2},\frac{|x|^2}{1+|x|^2} \right), \text{for any } x=(x_1,x_2)\in\R^2.
\end{equation*}
We know that $T$ is a conformal homeomorphism from $\R^2$ to $\mathcal{S}\backslash \{N\}$ where $N :=(0,0,1)$ is the North pole of $\mathcal{S}$. 

The procedure is as follows: Given $V:\R^2\to\R$,  we may define (see \cite{Hardy}) $\mathcal V:\mathcal S\to \R$  by letting
\begin{equation}\label{transport}\mathcal V(T(x)) = V(x) - \log(1+|x|^2), \quad \mathcal V(N) = \liminf_{|x|\to +\infty} \{V(x) - \log(1+|x|^2)\}.\end{equation}
Then   $V$ satisfies the hypothesis of Theorem~\ref{equi} if and only if $\mathcal V$  is a lower semicontinuous function on $\mathcal S$ which is finite on a nonpolar set. Therefore, in this case, the minimum of 
$$I_{\mathcal V}(\mu) := \iint_{\mathcal S\times\mathcal S}\left( -\log\|x-y\| +\frac{\mathcal V(x)}2 + \frac{\mathcal V(y)}2 \right)\,d\mu(x)\,d\mu(y)$$
among probability measures on $\mathcal S$ is achieved. Here $\|x-y\|$ denotes the euclidean norm in $\R^3$. Moreover, see \cite{Hardy}, the minimizer $\mu_{\mathcal V}$ is related to $\mu_V$ by the following relation
\begin{equation}\label{pullback}\mu_{\mathcal V} = T\#\mu_V,\end{equation}
where $T\#\mu$ denotes the push-forward of the measure $\mu$ by the map $T$. 

\label{ad}
\begin{defi} \label{admissiblef}
We say that $V:\R^2\to \R$ is \textbf{admissible} if it is of class $C^3$ and if, defining $\mathcal V$ as above, 
\begin{enumerate}
\item \textnormal{\textbf{(H1)}}: The set $\{x\in \R^2 ; V(x)<+\infty\}$ is not log-polar and $\displaystyle \liminf_{|x|\to +\infty}\{V(x)-\log(1+|x|^2)  \}>-\infty.$
\item \textnormal{\textbf{(H2)}}: The equilibrium measure $\mu_{\mathcal V}$ is of the form $m_{\mathcal V}(x)\indic_{\Sigma_\V}(x) \,dx$, where $m_{\mathcal V}$ is a $C^1$ function on $\mathcal S$ and $dx$ denotes the surface element on $\mathcal S$, where the function $m_{\mathcal V}$ is bounded above and below by positive constants $\bm$ and $\mb$, and where $\Sigma_{\V}$ is a compact subset of $\ms$ with $C^1$ boundary.
\end{enumerate}
\end{defi}

\begin{remark} Using (H2) and \eqref{pullback}, we find that 
$$ d\mu_V(x) = m_V(x) \indic_{\Sigma_V} \,dx,$$
where $\Sigma_V = T^{-1}(\Sigma_\V)$ and 
\begin{equation}\label{mv}m_V(x) = \frac{m_{\mathcal V}(T(x))}{(1+|x|^2)^2}.\end{equation}
Note that  $(1+|x|^2)^{-2}$ is the Jacobian of the transformation $T$.
\end{remark}

\section{Splitting Formula}\label{SpF}
Assume  $V$ is  admissible. We define as in \cite{2DSandier} the blown-up quantities:
\begin{equation*}
x'=\sqrt{n}x, \qquad m_V'(x')=m_V(x), \qquad d\mu_V'(x')=m_V'(x')dx'
\end{equation*}
and we define
\begin{equation}\label{zetadefinition}
\displaystyle \zeta(x):=U^{\mu_V}(x)+\frac{V(x)}{2}-c_V,
\end{equation}
where $c_V$ is the Robin constant given in \eqref{cvdef}. Then by \eqref{Frost1} and \eqref{Frost2}, $\zeta(x)=0$ q.e. in $\Sigma_V$ and $\zeta(x)\geq 0$ q.e. in $\R^2\backslash \Sigma_V$.\\
To any $n$-tuple of points $(x_1,...,x_n)\in (\R^2)^n$, we will now associate several quantities. First the probability measure 
\begin{equation}\label{nun}\nu_n=\frac1n\sum_{i=1}^n\delta_{x_i},\end{equation}
then the potential
\begin{equation}\label{defHn}
H_n:=-2\pi n \Delta^{-1}(\nu_n-\mu_V)=- n \int_{\R^2}\log|.-y|d(\nu_n-\mu_V)(y)=-\sum_{i=1}^n\log|.-x_i|-n U^{\mu_V}
\end{equation}
 where $\Delta^{-1}$ is the convolution operator with $\frac{1}{2\pi}\log|\cdot|$, hence such that $\Delta\circ\Delta^{-1}=\text{Id}$ where $\Delta$ denotes the usual Laplacian. We also define the rescaled measure (which is not a probability measure since it has mass $n$)
 \begin{equation}\label{nunp}\nu_n'=\sum_{i=1}^n\delta_{x_i'},\end{equation}
 and the rescaled potential
 \begin{equation}\label{hnp}
 H_n':=-2\pi\Delta^{-1}(\nu_n'-\mu_V').
 \end{equation}
Finally we will use the following notation for the associated electric field in rescaled coordinates
 \begin{equation}\label{enu}
 E_{\nu_n}:=-2\pi\Delta^{-1}(\nu_n'-\mu_V').
 \end{equation}
Note that even though $E_{\nu_n}$ is defined in rescaled variables, we do not use a prime in the notation to lighten notation.

 \begin{lemma}\label{convHn} Let $V$ be an admissible potential. Then we have
 \begin{equation}\label{cvHn}
 \lim_{R\to +\infty} \int_{B_R} H_n'(x)d\mu_V'(x)=\int_{\R^2} H_n'(x)d\mu_V'(x),\quad \lim_{R\to +\infty} W(\nabla H_n',\indic_{B_R})=W(\nabla H_n',\indic_{\R^2}).
 \end{equation}
 \end{lemma}
 \begin{proof} 
From Definition~\ref{defUmu} and \eqref{defHn} we have, letting $r = |x|$,  
 \begin{multline} \label{Hn}
 H_n(x)=\sum_{i=1}^n\int_{\R^2} \log\left( \frac{|x-y|}{|x-x_i|} \right)d\mu_V(y) = \\
 =\frac n2\int_{\R^2} \log\left(1-2\frac{x}{r}\cdot \frac yr + \frac{|y|^2}{r^2}\right)d\mu_V(y)- \frac12\sum_{i=1}^n \log\left(1-2\frac{x}{r}\cdot \frac{x_i}r + \frac{|x_i|^2}{r^2}\right).
 \end{multline}
 From \eqref{mv} we know that $d\mu_V(y) = m_V(y)\,dy$ where $|m_V(y)| < C/(1+|y|^2)^2$. By replacing in \eqref{Hn} and in the expression for $\nabla H_n(x)$ deduced frome \eqref{Hn} by differentiating, we easily deduce that if $|x| > R_0 := 2 \max_i |x_i|$ then 
 $$ |H_n(x)|\le \frac C{|x|},\quad  |\nabla H_n(x)|\le \frac C{|x|^2}.$$
Using \eqref{mv} again this implies   that $H_n\in L^1(\mu_V)$, hence the first equality in \eqref{cvHn}. This also implies that $|\nabla H_n|^2$ in in $L^1(\R^2\setminus B_{R_0})$. Then, since  
$$W(\nabla H_n,\indic_{\R^2})=W(\nabla H_n,\indic_{B_{R_0}})+\frac12 \int_{\R^2\setminus B_{R_0}}|\nabla H_n|^2,$$
the second equality in\eqref{cvHn} follows.
 \end{proof}
 
\begin{lemma} Let $V$ be admissible. Then, for every configuration $(x_1,...x_n)\in (\R^2)^n$, $n\geq 2$, we have
\begin{equation}
w_n(x_1,...,x_n)=n^2I_V(\mu_V)-\frac{n}{2}\log n +\frac{1}{\pi}W(E_{\nu_n},\indic_{\R^2})+2n\sum_{i=1}^n\zeta(x_i). \label{splitting}
\end{equation}
\end{lemma}

\begin{proof}
We may proceed as in the proof of \cite[Lemma 3.1]{2DSandier} and make use of the Frostman type inequalities \eqref{Frost1} and \eqref{Frost2} and Lemma \ref{convHn}. The important point is that, as shown in the proof of the previous lemma, we have $H_n(x)=O(|x|^{-1})$ and $\nabla H_n(x)=O(|x|^{-2})$ as $|x|\to +\infty$ which implies, exactly like in the case of compact support, that
$$
\lim_{R\to +\infty} \int_{\partial B_R} H_n(x)\nabla H_n(x).\vec{\nu}(x)dx=0
$$
where $\vec{\nu}(x)$ is the outer unit normal vector at $x\in \partial B_R$.
\end{proof}

\section{Lower bound}\label{UsLem}
Here we follow the strategy of \cite{2DSandier}, pointing out the required modifications in the noncompact case.

\subsection{Mass spreading result and modified density $g$}
We have the following result  from \cite[Proposition 3.4]{2DSandier}:
\begin{lemma} Let $V$ be admissible and assume  $(\nu,E)$ are such that $\displaystyle \nu=\sum_{p\in\Lambda}\delta_p$ for some finite subset $\Lambda \subset \R^2$ and $\divergence E=2\pi(\nu -m_V)$, $\curl E=0$ in $\R^2$. Then, given any $\rho>0$ there exists a signed measure $g$ supported on $\R^2$ and such that:
\begin{itemize}
\item There exists a family $\mathcal{B}_\rho$ of disjoint closed balls covering $\supp(\nu)= \Lambda$ such that the sum of the radii of the balls intersecting any ball of radius
1 is bounded by  $\rho$~; furthermore,
\begin{equation*}
g(A)\geq -C(\|m_V\|_\infty+1)+\frac{1}{4}\int_A|E(x)|^2\indic_{\Omega\backslash\mathcal{B}_\rho}(x)dx, \quad  \text{ for any } A\subset\R^2,
\end{equation*}
where $C$ depends only on $\rho$;\\
\item we have
\begin{equation*} dg(x)=\frac{1}{2}|E(x)|^2dx \quad \text{ outside } \bigcup_{p\in\Lambda}B(p,\lambda)
\end{equation*}
where $\lambda$ depends only on $\rho$;\\
\item there exists $\lambda, C>0$ depending only on $\rho$ such that for any function $\chi$ compactly supported in $\R^2$ we have
\begin{equation*}
\left|W(E,\chi)-\int\chi dg  \right|\leq CN(\log N +\|m_V\|_\infty)\|\nabla\chi\|_\infty
\end{equation*}
where $N=\string #\{p\in\Lambda ; B(p,\lambda)\cap \supp(\nabla\chi)\neq\emptyset  \}$;\\
\item for any $U\subset \Omega$
\begin{equation*}\string #(\Lambda \cap U)\leq C(1+\|m_V\|^2_\infty|\hat{U}|+g(\hat{U}))
\end{equation*}
where $\hat{U}:=\{x\in \R^2 ; d(x,U)<1\}$.
\end{itemize}
\end{lemma}

\begin{defi} \label{massdefi} Assume $\displaystyle \nu_n=\frac1n\sum_{i=1}^n\delta_{x_i}$. Defining $ \nu'_n$  and $E_{\nu_n}$ as in \eqref{nunp} and \eqref{enu}, we denote by $g_{\nu_n}$ the result of applying the previous Lemma to $(\nu'_n,E_{\nu_n})$.
\end{defi}
The following result  \cite[Lemma 3.7]{2DSandier} connects $g$ and the renormalized energy.
\begin{lemma} \textnormal{(\cite{2DSandier})} For any $\displaystyle \nu_n=\sum_{i=1}^n\delta_{x_i}$, we have
\begin{equation} \label{Wg} W(E_{\nu_n},\indic_{\R^2})=  \int_{\R^2}dg_{\nu_n}.
\end{equation}
\end{lemma}

\subsection{Ergodic Theorem}
We adapt the abstract setting in \cite[Section 4.1]{2DSandier}.  We are given a Polish space $X$, which is a space of functions, on which $\R^2$ acts continuously. 
We denote this action $(\lambda, u)\to \theta_\lambda u:=u(.+\lambda)$, for any $\lambda\in\R^2$ and $u\in X$. We assume it is continuous with respect to both $\lambda$ and $u$.

We also define $T^\varepsilon_\lambda$ and $T_\lambda$ acting on $\R^2\times X$, by $T^\varepsilon_\lambda(x,u):=(x+\varepsilon\lambda,\theta_\lambda u)$ and $T_\lambda(x,u):=(x,\theta_\lambda u)$.

For a probability measure $P$ on $\R^2\times X$ we say that $P$ is $T$-invariant if for every $\lambda\in\R^2$, it is invariant under the mapping $(x,u)\mapsto (x,\theta_{\lambda}u)$.

We let  $\{f_\ep\}_\ep$, and $f$ be measurable functions defined on $\R^2\times X$ which satisfy the following properties. For any  sequence $\{x_\varepsilon,u_\varepsilon\}_\varepsilon$ such that $x_\varepsilon \to x$ as $\varepsilon\to 0$  and such that for any $R>0$, 
$$
\limsup_{\varepsilon \to 0} \int_{B_R}f_\varepsilon(x_\varepsilon+\varepsilon\lambda,\theta_\lambda u_\varepsilon)d\lambda <+\infty,
$$
we have 
\begin{enumerate}
 \item (Coercivity) $\{u_\varepsilon\}_\varepsilon$ has a convergent subsequence;
 \item ($\Gamma$-liminf) If $\{u_\varepsilon\}_\varepsilon$ converges to $u$, then $\displaystyle \liminf_{\varepsilon\to 0}f_\varepsilon (x_\varepsilon, u_\varepsilon)\geq f(x,u)$.
 \end{enumerate}
\begin{remark}
In contrast with the compact case, not every sequence $\{x_\varepsilon\}$ has a convergent subsequence, hence convergence needs to be assumed.
\end{remark}

Now let $V$ be an admissible potential on $\R^2$ and $\mu_V$ its associated equilibrium measure.  We have 

\begin{thm} \label{erg} Let $V$, $X$, $(f_\varepsilon)_\varepsilon$ and $f$ be as above. We define 
$$
\displaystyle F_\varepsilon(u):=\int_{\R^2}f_\varepsilon(x,\theta_{\frac{x}{\varepsilon}}u)d\mu_V(x)
$$
Assume  $(u_\varepsilon)_\varepsilon\in X$ is  a sequence such that $F_\varepsilon(u_\varepsilon)\leq C$ for any $\varepsilon>0$. Let $P_\varepsilon$ be the image of $\mu_V$ by $x\mapsto (x,\theta_{\frac{x}{\varepsilon}}u_\varepsilon)$, then:
\begin{enumerate}
\item $(P_\varepsilon)_\varepsilon$ admits a convergent subsequence to a probability measure $P$,
\item the first marginal of $P$ is $\mu_V$,
\item $P$ is $T$-invariant,
\item for $P-a.e.$ $(x,u)$, $(x,u)$ is of the form $\displaystyle\lim_{\varepsilon\to 0}(x_\varepsilon,\theta_{\frac{x_\varepsilon}{\varepsilon}}u_\varepsilon)$,
\item $\displaystyle \liminf_{\varepsilon\to 0} F_\varepsilon(u_\varepsilon)\geq \int_{\R^2\times X} f(x,u) dP(x,u)$.
\item Moreover we have
\begin{equation*}\int_{\R^2\times X} f(x,u)dP(x,u)=\int_{\R^2\times X}\left(\lim_{R\to +\infty}\dashint_{B_R}f(x,\theta_\lambda u)d\lambda \right)dP(x,u). 
\end{equation*}
where $\dashint_{B_R}$ denote the integral average over $B_R$.
\end{enumerate}
\end{thm}

\begin{proof} The proof follows \cite{Sandier_Serfaty,2DSandier} but with $\mu_V$ replacing  the normalized Lebesgue measure on a compact set  $\Sigma_V$. We sketch it and detail the parts where modifications are needed. For any $R>0$ we let $\mu_V^R$ denote the restriction of $\mu_V$ to $B_R$, and $ P_\ep^R$ denote the image of $\mu_V^R$  by the map $x\mapsto (x,\theta_{\frac{x}{\varepsilon}}u_\varepsilon)$.\medskip

\noindent\textbf{Step 1: Convergence of a subsequence of $(P_\varepsilon)$ to a probability measure $P$.} It suffices to prove that the sequence $\{P_\ep\}_\ep$ is tight. From  \cite{Sandier_Serfaty,2DSandier}, which deals with the compact case, $\{P_\ep^R\}_\ep$ is tight, for any $R>0$.

Now take any $\delta>0$, we need to prove that there exists a compact subset $K_\delta$ of $\R^2\times X$ such that $P_\ep(K_\delta)>1-\delta$ for any $\ep>0$. For this we first choose $R>0$ large enough so that $(\mu_V - \mu_V^R)(\R^2)<\delta/2$. This implies that $P_\ep^R$ has total measure at least $1-\delta/2$ and then we may use the tightness of $\{P_\ep^R\}_\ep$ to find that there exists a compact set $K_\delta$ such that $P_\ep^R(K_\delta)>1-\delta$. It follows that $P_\ep(K_\delta)>1-\delta$, and then that $\{P_\ep\}_\ep$ is tight.\medskip

\noindent \textbf{Step 2: $P$ is $T$-invariant.} Let $\lambda\in\R^2$, let  $\Phi$ be a bounded continuous function on $\R^2\times X$ and let $P_\lambda$ be the image of $P$ by $(x,u)\mapsto (x,\theta_{\lambda}u)$. By the change of variables $y=\varepsilon\lambda+x=(\varepsilon\lambda+I_2)(x)$ and   for a subsequence $\varepsilon\to 0$ along which $P_\varepsilon\to P$, we obtain,
\begin{align*}
\int_{\R^2\times X}\Phi(x,u)dP_\lambda(x,u) &=\int_{\R^2\times X}\Phi(x,\theta_{\lambda}u)dP(x,u)\\
&= \lim_{\varepsilon\to 0}\int_{\R^2\times X}\Phi(x,\theta_{\lambda}u)dP_\varepsilon(x,u)\\
&= \lim_{\varepsilon\to 0}\int_{\R^2}\Phi(x,\theta_{\lambda+\frac{x}{\varepsilon}}u_\varepsilon)d\mu_V(x)\\
&= \lim_{\varepsilon\to 0}\int_{\R^2}\Phi(x,\theta_{\frac{\varepsilon\lambda+x}{\varepsilon}}u_\varepsilon)d\mu_V(x)\\
&= \lim_{\varepsilon\to 0}\int_{\R^2}{\Phi\left(y-\lambda\varepsilon,\theta_{\frac{y}{\varepsilon}}u_\varepsilon  \right)m_V(y-\lambda\varepsilon)dy}.
\end{align*}

From the boundedness of  $\Phi$ and  the decay properties of $m_V$ (see \eqref{mv}) it is straight forward to check  that, along the same subsequence $\varepsilon\to 0$,
\begin{align*}
\lim_{\varepsilon\to 0}\int_{\R^2}{\Phi\left(y-\lambda\varepsilon,\theta_{\frac{y}{\varepsilon}}u_\varepsilon  \right)m_V(y-\lambda\varepsilon)dy}
 &= \lim_{\varepsilon\to 0}\int_{\R^2}\Phi\left(y-\lambda\varepsilon,\theta_{\frac{y}{\varepsilon}}u_\varepsilon  \right) m_V(y)\,dy\\
&=\lim_{\varepsilon\to 0}\int_{\R^2\times X}\Phi\left(y-\lambda\varepsilon,u\right)dP_\varepsilon(y,u).
\end{align*}
Then, arguing as in \cite{2DSandier} using the tightness of $(P_\varepsilon)_\varepsilon$ we obtain 
\begin{align*}
\lim_{\varepsilon\to 0}\int_{\R^2\times X}\Phi\left(y-\lambda\varepsilon,u\right)dP_\varepsilon(y,u)&=\lim_{\varepsilon\to 0}\int_{\R^2\times X}\Phi(y,u)dP_\varepsilon(y,u)\\
&=\int_{\R^2\times X}\Phi(x,u)dP(x,u),
\end{align*}
which concludes the proof that $\displaystyle \int_{\R^2\times X}\Phi(x,u)dP_\lambda(x,u) = \int_{\R^2\times X}\Phi(x,u)dP(x,u)$, i.e. that $P$ is $T$-invariant.\medskip

\noindent Items 2 and 4 in the theorem are obvious consequences of the definition of $P$ and  items $5$ and  $6.$ require no modification from \cite{2DSandier}. We have proved above items 1 and 3.
\end{proof}

\section{Asymptotic Expansion of the Hamiltonian}\label{AEOTH}
We define \begin{equation*}\alpha_V:=\frac{1}{\pi}\int_{\R^2}\min_{\mathcal{A}_{m_V(x)}}Wdx=\frac{1}{\pi}\min_{\mathcal{A}_1}W-\frac{1}{2}\int_{\R^2}m_V(x)\log m_V(x)dx,\end{equation*}
where the equality is a consequence of  \eqref{scaling}. The fact  that $\alpha_V$ is finite follows from     \eqref{mv}, which ensures that the integral converges.

Recalling the notations \eqref{zetadefinition}-\eqref{enu}, we define
\[ F_n(\nu)=
\left\{
\begin{array}{ll}
\displaystyle \frac{1}{n}\left(\frac{1}{\pi}W(E_n,\indic_{\R^2})+2n\int \zeta d\nu  \right) &\mbox{if $\nu$ is of the form $\displaystyle \frac1n\sum_{i=1}^n\delta_{x_i}$,}\\
+\infty &\mbox{otherwise.}
\end{array}
\right. \]
and set, for any measure $\nu_n$ of the form \eqref{nun}, 
\begin{equation*} 
P_{\nu_n}:=\int_{\R^2}\delta_{(x,E_{\nu_n}(x\sqrt{n}+.))}d\mu_V(x).
\end{equation*}
The following result extends  \cite[Theorem 2]{2DSandier} to a  class of equilibrium measures with possibly unbounded support, which requires a restatement which makes it slightly different from its counterpart in \cite{2DSandier}. It is essentially  a Gamma-Convergence  (see \cite{BraidesGamma}) statement, consisting  of  a lower bound and an upper bound, the two implying 
the convergence of $\displaystyle \frac{1}{n}\left[w_n(x_1,...,x_n)-n^2I_V(\mu_V)+\frac{n}{2}\log n\right]$ to $\alpha_V$ for a minimizer $(x_1,...,x_n)$ of $w_n$.

\subsection{Main result}
\begin{thm} \label{THM} Let $1<p<2$ and $X=\R^2\times L_{loc}^p(\R^2,\R^2)$. Let $V$ be an admissible function.\\
\textnormal{\textbf{A. Lower bound:}} Let $(\nu_n)_n$ such that $F_n(\nu_n)\leq C$, so that in particular $\nu_n$ is of the form \eqref{nun} for every $n$. Then:
\begin{enumerate}
\item $P_{\nu_n}$ is a probability measure on $X$ and  $(P_{\nu_n})_n$ admits a subsequence which converges to a probability measure $P$ on $X$,
\item the first marginal of $P$ is $\mu_V$,
\item $P$ is $T$-invariant,
\item $E\in \mathcal{A}_{m_V(x)}$ for  $P$-a.e. $(x,E)$,
\item we have the lower bound 
\begin{equation}\label{LB}
 \liminf_{n\to +\infty}\frac{1}{n\pi}W(E_{\nu_n},\indic_{\R^2}) \geq \frac{1}{\pi}\int_{\R^2} \frac{W(E)}{m_V(x)}dP(x,E)\geq \alpha_V.
 \end{equation}
\end{enumerate}
\textnormal{\textbf{B. Upper bound.}} Conversely, assume $P$ is a $T$-invariant probability measure on $X$ whose first marginal is $\mu_V$ and such that for $P$-almost every $(x,E)$ we have $E\in \mathcal{A}_{m_V(x)}$. Then there exist a sequence $\{\nu_n=\sum_{i=1}^n\delta_{x_i}\}_n$ of measures on $\R^2$ and a sequence $\{E_n\}_n$ in $L_{loc}^p(\R^2,\R^2)$ such that $\divergence E_n=2\pi(\nu'_n-m_V')$ and such that, defining $\displaystyle P_n=\int_{\R^2}\delta_{(x,E_n(x\sqrt{n}+.))}d\mu_V(x)$, we have $P_n \to P$ as $n\to +\infty$ and
\begin{equation}\label{UB}
\limsup_{n\to +\infty} F_n(\nu_n)\leq \frac{1}{\pi}\int_{\R^2} \frac{W(E)}{m_V(x)}dP(x,E).
\end{equation}
\textnormal{\textbf{C. Consequences for minimizers.}} Let for any $n$,  $(x_1,...,x_n)$ denote a minimizer of  $w_n$ and let $\displaystyle \nu_n=\sum_{i=1}^n\delta_{x_i}$. Then, for any weak subsequential limit $P$ of $(P_{\nu_n})_n$ we have:
\begin{enumerate}
\item for $P$-almost every $(x,E)$, $E$ minimizes $W$ over $\mathcal{A}_{m_V(x)}$;
\item we have 
\begin{equation*} 
\lim_{n\to +\infty} F_n(\nu_n)=\lim_{n\to +\infty} \frac{1}{n\pi}W(E_{\nu_n},\indic_{\R^2}) = \frac{1}{\pi}\int_{\R^2} \frac{W(E)}{m_V(x)}dP(x,E)= \alpha_V,
\end{equation*}
hence we obtain the following asymptotic expansion, as $n\to +\infty$:
\begin{equation*}  \min_{(\R^{2})^n} w_n=I_V(\mu_V)n^2-\frac{n}{2}\log n +\alpha_V n + o(n).
\end{equation*}
\end{enumerate}
\end{thm}

\subsection{Proof of the lower bound}
We follow the same lines as in \cite[Section 4.2]{2DSandier}. Because  $F_n(\nu_n)\leq C$ and \eqref{splitting}, we have that $$\frac1{n^2} w_n(x_1,\dots x_n)\to I_V(\mu_V),$$ therefore $\nu_n$ converges to $\mu_V$ (this follows from the results in \cite{Hardy}).

We let $\nu'_n = \sum_i \delta_{x_i'}$, and $E_{\nu_n}$, $H_n'$, $g_n$ be as in  Definition \ref{massdefi}.

Let $\chi$ be a $C^\infty$ cutoff function supported on  the unit ball $B_1$ and with integral equal to 1. We define
\begin{equation*}
\mathbf{f}_n(x,\nu,E,g):=
\left\{ 
       \begin{array}{ll}
       \displaystyle\frac{1}{\pi}\int_{\R^2}\frac{\chi(y)}{m_V(x)}dg(y) & \mbox{if $(\nu,E,g)=\theta_{\sqrt{n}x}({\nu}'_n,E_{\nu_n},{g}_n)$,} \\
       +\infty & \mbox{otherwise.}
       \end{array}
       \right.
\end{equation*}
As in \cite[Section 4.2, Step 1]{2DSandier}, if we let 
\begin{equation*}
\mathbf{F}_n(\nu,E,g):=\int_{\R^2}\mathbf{f}_n\left(x, \theta_{x\sqrt{n}}(\nu,E,g)\right)d\mu_V(x),\end{equation*}
then 
\begin{align*}\mathbf F_n(\nu_n', E_{\nu_n},g_n) 
= \int_{\R^2}\frac{1}{\pi}\int_{\R^2}\frac{\chi(y)}{m_V(x)}d(\theta_{x\sqrt{n}}\# g)d\mu_V(x) &= 
\frac{1}{\pi}\int_{\R^2}\int_{\R^2}\chi(y-x\sqrt{n})dxd{g}_n(y)\\
&\leq \frac{1}{n\pi}W(E_{\nu_n},\indic_{\R^2})+\frac{{g}_n^{-}(U^c)}{n\pi},
\end{align*}
by \eqref{Wg}, where $\displaystyle U=\{x':d(x',\R^2\backslash \Sigma)\geq 1\}$. As in \cite{2DSandier}, we have ${g}_n^{-}(U^c)=o(n)$. Hence, if $(\nu,E,g)=({\nu}'_n,E_{\nu_n},{g}_n)$, as $n\to +\infty$:
\begin{equation*} \mathbf{F}_n(\nu,E,g)\leq \frac{1}{n\pi}W(E_{\nu_n},\indic_{\R^2})+o(1),
\end{equation*}
and $\mathbf{F}_n(\nu,E,g)=+\infty$ otherwise.\\

Now, as in \cite{2DSandier}, we want to use Theorem~\ref{erg} with $\varepsilon=\frac{1}{\sqrt{n}}$ and $X=\mathcal{M}_+\times L_{loc}^p(\R^2,\R^2)\times \mathcal{M}$ where $p\in]1,2[$, $\mathcal{M}_+$ is the set of nonnegative Radon measures on $\R^2$ and $\mathcal{M}$ the set of Radon measures bounded below by $-C_V:=-C(\|m_V\|_\infty^2+1)$. Let $Q_n$ be the image of $\mu_V$ by $x\mapsto (x,\theta_{x\sqrt{n}}(\bar{\nu}'_n,E_{\nu_n},{g}_n))$. We have:\\ 

1) The fact that $\textbf{f}_n$ is coercive is proved as in \cite[Lemma 4.4]{2DSandier}. Indeed, if $(x_n,\nu(n),E(n),g(n))_n$ is such that $x_n\to x$ and, for any $R>0$,
\begin{equation*} \limsup_{n\to +\infty}\int_{B_R}\textbf{f}_n\left(x_n+\frac{\lambda}{\sqrt{n}},\theta_\lambda(\nu(n),E(n),g(n))\right)d\lambda <+\infty,
\end{equation*}
then the integrand is bounded for a.e. $\lambda$. By assumption on $\textbf{f}_n$, for any $n$,
$$
\theta_\lambda(\nu(n),E(n),g(n))=\theta_{x_n\sqrt{n}+\lambda}({\nu}_n',E_{\nu_n},{g}_n),
$$
hence it follows that
\begin{equation*}
(\nu(n),E(n),g(n))=\theta_{x_n\sqrt{n}}({\nu}_n',E_{\nu_n},{g}_n).
\end{equation*}
For any $R>0$, there exists $C_R>0$ such that for any $n>0$, 
\begin{align*}
\int_{B_R}\textbf{f}_n\left(x_n+\frac{\lambda}{\sqrt{n}},\theta_\lambda(\nu_n,E_{\nu_n},g_n)\right)d\lambda &=\int_{B_R}\frac{1}{\pi}\int_{\R^2}\frac{\chi(y)}{m_V\left(x_n+\frac{\lambda}{\sqrt{n}}\right)}d(\theta_{\lambda+x_n\sqrt{n}}\# {g}_n(y))d\lambda\\
&=\frac{1}{\pi}\int_{B_R}\int_{\R^2}\frac{\chi(y-x_n\sqrt{n}-\lambda)}{m_V\left(x_n+\frac{\lambda}{\sqrt{n}}\right)}d{g}_n(y)d\lambda\\
&=\frac{1}{\pi}\int_{\R^2}\chi\ast\left(\indic_{B_R(x_n\sqrt{n})}\frac{1}{m_V(./\sqrt{n})}  \right)(y)d{g}_n(y)<C_R.
\end{align*}
This, inequalities \eqref{mv} and the fact that ${g}_n$ is bounded below imply that ${g}_n(B_R(x_n\sqrt{n}))$ is bounded independently of $n$. Hence by the same argument as in  \cite[Lemma 4.4]{2DSandier}, we have the convergence of a subsequence of $(\nu(n),E(n),g(n))$.\\ 

2) We have the $\Gamma$-liminf property: if $(x(n),\nu(n),E(n),g(n))\to (x,\nu,E,g)$ as $n\to +\infty$, then, by Fatou's Lemma,
\begin{equation*}\liminf_{n\to +\infty} \mathbf{f}_n(x(n),\nu(n),E(n),g(n))\geq \mathbf f(x,\nu,E,g):=\frac{1}{\pi}\int\frac{\chi(y)}{m_V(x)} dg(y),
\end{equation*}
obviously if the left-hand side is finite. Therefore, Theorem \ref{erg} applies and implies that:\\
\begin{enumerate}
\item The sequence of measures $(Q_n)_n$ admits a subsequence which converges to a measure $Q$ which has $\mu_V$ as  first marginal.
\item  It holds that $Q$-almost every $\displaystyle (x,\nu,E,g)$ is of the form $\displaystyle \lim_{n\to +\infty}(x_n, \theta_{x_n\sqrt{n}}({\nu}'_n,E_{\nu_n},{g}_n))$.
\item The measure $Q$ is $T$-invariant.
\item We have $\displaystyle \liminf_{n\to +\infty} \mathbf{F}_n({\nu}'_n,E_{\nu_n},{g}_n)\geq \frac{1}{\pi} \int_{\R^2}\left(\int_{\R^2}\frac{\chi(y)}{m_V(x)} dg(y)\right) dQ(x,\nu,E,g)$.
\item $\displaystyle \frac{1}{\pi}\int \int \frac{\chi(y)}{m_V(x)}dg(y) dQ(x,\nu,E,g)=\int \left( \lim_{R\to +\infty}\dashint_{B_R} \int \frac{\chi(y-\lambda)}{m_V(x)}dg(y)d\lambda\right) dQ(x,\nu,E,g)$.
\end{enumerate}
Now we can follow exactly \cite[24, Section 4.2, Step 3]{2DSandier}. We
notice that $P_n$ is the marginal of $Q_n$ corresponding to the variables $(x,E)$, and deduce
from 4) that 
\begin{align*}
\liminf_{n\to+\infty} \frac{1}{n\pi}W(E_{\nu_n},\indic_{\R^2})&\geq\int \left( \int \chi dg \right)\frac{dQ(x,\nu,E,g)}{m_V(x)} \\
&= \int\lim_{R\to +\infty} \left( \frac{1}{\pi R^2}\int \chi\ast \indic_{B_R} dg\right)\frac{dQ(x,\nu,E,g)}{m_V(x)}\\
&\geq \frac{1}{\pi}\int W(E) \frac{dQ(x,\nu,E,g)}{m_V(x)}=\frac{1}{\pi}\int \frac{W(E)}{m_V(x)} dP(x,E).
\end{align*}
Thus the lower bound \eqref{LB} is proved. The fact that the right-hand side is larger than $\alpha_V$ is obvious because the first marginal of $\displaystyle \frac{dP}{m_V}$ is the Lebesgue measure.

\subsection{Proof of the upper bound, the case $\supp(\mu_V)\neq\R^2$} The discussion following Theorem~\ref{equi} permits to  reduce the case of  $V$'s such that $\supp (\mu_V)\neq \R^2$ to the case of a compact support. We now explain how this is done.

Since  $\supp (\mu_V)\neq \R^2$, there exists $y\in\ms$ which does not belong to the support of $\mu_\V$. Let $R$ be a rotation such that $R(N) = y$, then the minimum of  $I_{\V\circ R}$ is $\mu_{\V\circ R} = R^{-1}\#\mu_\V$ hence $N$ does not belong to its support. 

Letting  $\varphi = T^{-1} R T$, we have that $\varphi$  is of the form $z\to\ds \frac{az+b}{cz+d}$ with $ad-bc = 1$, and applying \eqref{transport}, \eqref{pullback} to $\V\circ R$ we have that 
$$\mu_{V_\varphi} = T^{-1}\#\mu_{\V\circ R},$$
where 
$$\mathcal V\circ R(T(x)) = V_\varphi(x) - \log(1+|x|^2).$$
This implies that $\mu_{V_\varphi}$ has compact support since $N$ does not belong to the support of $\mu_{\V\circ R} $. Moreover,  using \eqref{transport} again to evaluate $\V(R T(x))$ we find for any $x$ such that $RT(x)\neq N$, i.e. $x\neq -d/c$, 
\begin{align*}
&V_\varphi(x)  = V(T^{-1}RT(x)) - \log(1+|T^{-1}RT(x)|^2)  + \log(1+|x|^2),\\
& V_\varphi(-d/c) = \V(N) + \log(1+|d/c|^2)=\log(1+|d/c|^2) +\liminf_{|x|\to +\infty} \{V(x) - \log(1+|x|^2)\}.
\end{align*}
Finally we find that 
\begin{equation}\label{vp} V_\varphi(x)  = V(\varphi(x)) - \log(1+|\varphi(x)|^2) + \log(1+|x|^2), \quad V_\varphi(-d/c) = \liminf_{y\to-d/c} V_\varphi(y).
\end{equation}

Now we rewrite the discrete energy by changing variables, to find that, writing $w_{n,V}$ instead of $w_n$ to clarify the dependence on $V$, 
\begin{equation}\label{disc1}w_{n,V}(x_1,\dots,x_n) = -\sum_{i\neq j}^n\log|\varphi(y_i) - \varphi(y_j)|+ n\sum_{i=1}^nV(\varphi(y_i)),\end{equation}
where $x_i = \varphi(y_i)$. Now we use the identity (see \cite{Electsphere}, \cite{Hardy})
$$\|T(x)-T(y)\|=\frac{|x-y|}{\sqrt{1+|x|^2}\sqrt{1+|y|^2}}$$
applied to $\varphi( x)$, $\varphi(y)$ together with the fact that  $\varphi = T^{-1} R T$ and that $R$ is a rotation to get
$$\|T(x)-T(y)\|=\frac{|\varphi(x)-\varphi(y)|}{\sqrt{1+|\varphi(x)|^2}\sqrt{1+|\varphi(y)|^2}}.$$
The two together imply that 
$$\log|\varphi(x) - \varphi(y)| = \log|x-y| + \frac12 \log (1+|\varphi(x)|^2) + \frac12 \log (1+|\varphi(y)|^2) - \frac12 \log (1+|x|^2)- \frac12 \log (1+|y|^2).$$
Replacing in \eqref{disc1} shows that 
\begin{equation}\label{chvar}w_{n,V}(x_1,\dots, x_n) = w_{n,V_\varphi}(y_1,\dots, y_n) + \sum_i \log (1+|\varphi(y_i)|^2)  - \sum_i \log (1+|y_i|^2),\quad x_i = \varphi(y_i).\end{equation}

It follows from \eqref{chvar} that an upper bound  for $\min w_{n,V}$ can be computed by using a minimizer for  $w_{n,V_\varphi}$ as a test function. But now we recall that $\mu_{V_\varphi}$ has compact support, hence the results of \cite{2DSandier} apply and we find, using the fact that for such a minimizer $\frac1n\sum_i\delta_{y_i}$ converges to $\mu_{V_\varphi}$,
\begin{equation}\label{mm}\min w_{n,V} \le  n^2 I_{V_\varphi}(\mu_{V_\varphi}) - \hal n \log n + \left(\alpha_{V_\varphi}  + \int \log \left(\frac{1+|\varphi(x)|^2}{1+|x|^2}\right)\,d\mu_{V_\varphi}(x)\right)n + o(n), \end{equation} 
where 
$$\alpha_{V_\varphi} = \frac{\alpha_1}{\pi}-\hal \int_{\Sigma_{V_\varphi}} m_{V_\varphi}(x)\log  m_{V_\varphi}(x)\, dx, \quad  \alpha_1:=\min_{\mathcal A_1} W.$$
We remark that $I_{V_\varphi}(\mu_{V_\varphi})=I_V(\mu_V)$ because $\mu_{V_\varphi}=\varphi^{-1}\# \mu_V$. Moreover, it follows from \eqref{mv} that 
$$m_{V_\varphi}(x) = m_V(\varphi(x))\left(\frac{1+|\varphi(x)|^2}{1+|x|^2}\right)^2,$$
which plugged in the expression for $\alpha_{V_\varphi}$ and then in \eqref{mm} yields, 
$$\min w_{n,V} \le n^2 I_V(\mu_V) - \hal n \log n + \alpha_{V} n +o(n), $$
which matches the lower-bound we already obtained and thus proves Theorem~\ref{thintro} in the case where the support of $\mu_V$ is not the full plane.

\subsection{Proof of the upper bound by compactification and conclusion}
Here we assume that $\Sigma_V=\R^2$.  Let
$$
\varphi(z):=-\frac{1}{z}=\varphi^{-1}(z).
$$
Then, using the notations of the previous section, we deduce from \eqref{vp} that 
$$
V_\varphi(z)=V(\varphi(z))+2\log|z|.
$$
To simplify exposition and notation, we assume that $\mu_V(B_1)=\mu_V(B_1^c)=1/2$, otherwise  there would exist $R$ such that $\mu_V(B_R)=\mu_V(B_R^c)=1/2$ and we should use the transformation $\varphi_R(z)=\varphi_R^{-1}(z)=-Rz^{-1}$ instead.

Our idea is to cut $\Sigma_V=\R^2$ into two parts in order to construct a sequence of $2n$ points associated to a sequence of vector-fields. We will only construct test configurations with an even number of points, again to simplify exposition and avoid unessential technicalities.\medskip

\noindent \textbf{Step 1: Reminder  of the compact case and notations.}
We reproduce below \cite[Corollary 4.6]{2DSandier} when $K$ is a compact set of $\R^2$. Note that we have replaced there the hypothesis of $T_{\lambda(x)}$-invariance (which is part of the definition of an admissible $P$) by the usual translation invariance. We give in the appendix a direct proof that the two notions are in fact equivalent, it would follow also  from the fact that the translation invariance implies that the disintingration measures are themselves invariant (see \cite[Remark~2.4]{leble}).

\begin{thm} \label{UpBound} \textnormal{(\cite{2DSandier})} Let $P$ be a $T$-invariant probability measure on $X=K\times L_{loc}^p(\R^2,\R^2)$, where $K$ is a compact subset of $\R^2$ with $C^1$ boundary. 

We assume  that $P$ has  first marginal $dx_{|K}/|K|$ and  that  for $P$-almost every $(x,E)$ we have $E\in\mathcal{A}_{m(x)}$, where $m$ is a smooth function on $K$ bounded above and below by positive constants. Then there exists a sequence $\{\nu_n=\sum_{i=1}^n\delta_{x_i}\}_n$ of empirical measures on $K$ and a sequence $\{E_n\}_n$ in $L_{loc}^p(\R^2,\R^2)$ such that $\divergence E_n=2\pi(\nu'_n-m')$, such that $E_n = 0$ outside $K$ and such that $\displaystyle P_{n}:=\dashint_{K}\delta_{(x,E_n(\sqrt{n}x+.))}dx  \to P$ as $n\to +\infty$. Moreover 
\begin{equation*}
\limsup_{n\to +\infty}      \frac{1}{n\pi} W(E_n,\indic_{\R^2})\leq \frac{|K|}{\pi}\int W(E)dP(x,E).\\ 
\end{equation*}
\end{thm}

We write $\mu_V=\mu_V^{(1)}+\mu_V^{(2)}$ where  $\mu_V^{(1)}:={\mu_V}_{|B_1}$ and $\mu_V^{(2)}:={\mu_V}_{|\overline B_1^c}$, where $\overline A $ denotes the closure of set $A$ in $\R^2$. Let $\tilde{\mu}_V^{(2)}:=\varphi\#\mu_V^{(2)}$, then we have 
$$
d\mu_V^{(1)}(x)=m_V(x)\indic_{B_1}(x)dx=:m_V^{(1)}(x)dx \quad \text{ and } \quad d\tilde{\mu}_V^{(2)}(x)=m_{V_\varphi}(x)\indic_{B_1}(x)dx=:m_{V_\varphi}^{(2)}(x)dx,
$$
where $m_{V_\varphi}(x)=m_V(\varphi^{-1}(x))| \det(D\varphi_x^{-1})|$.\\
Note that, by assumption \textbf{(H2)} and \eqref{mv} we have that there exists positive constants $\bm$ and $\mb$ such that, for any $x\in B_1$, 
$$
0<\mb\leq m_V(x)\leq \bm\quad \text{ and } \quad 0<\mb\leq m_{V_\varphi}(x)\leq \bm.
$$
Moreover the  boundary $\partial B_1$ is $C^1$. 

Now let   $P$ be a $T$-invariant probability measure on $X$ whose first marginal is $\mu_V$ and be such that for $P$-almost very $(x,E)$, we have $E\in \mathcal{A}_{m_V(x)}$. We can write 
$$P=P^{(1)}+P^{(2)},
$$ 
where $P^{(1)}$ is the restriction of $P$ to $B_1\times L^p_{loc}(\R^2,\R^2)$ with first marginal $\mu_V^{(1)}$, and $P^{(2)}$ is the restriction of $P$ to $B_1^c\times L^p_{loc}(\R^2,\R^2)$ with first marginal $\mu_V^{(2)}$. We define $\tilde P^{(1)}$ by the relation
$$
dP^{(1)}(x,u)=m_V(x)|B_1|d\tilde{P}^{(1)}(x,u),
$$
and then $\tilde{P}^{(1)}$ is a $T$-invariant probability measure on $B_1\times L^p_{loc}(\R^2,\R^2)$ with first marginal $dx_{|B_1}/|B_1|$ and such that, for $\tilde{P}^{(1)}$-a.e. $(x,E)$, $E\in \mathcal{A}_{m_V^{(1)}(x)}$. We denote by $\varphi\# P^{(2)}$ the pushforward of $P^{(2)}$ by 
\begin{equation}\label{chgtvar}
(x,E)\mapsto \left(y, \tilde{E}\right),\ \text{ where } \ y:=\varphi(x) \text{ and } \tilde{E}:=(D\varphi_y)^T E(D\varphi_y \cdot),
\end{equation}
where $D\varphi_x$ is the differential of $\varphi$ at point $x$. Then if $\div E   = 2\pi(\nu - m_V(x)dx)$ we have  $\div \tilde{E} = 2\pi(\vp\#\nu -|\partial_z\varphi(y)|^2 m_V(\varphi(y)))$ so that for $\varphi\# P^{(2)}$-a.e. $(y,\tilde{E})$ the vector field $\tilde{E}$ belongs to  $ \mathcal{A}_{m_{V_\varphi}(y)}$, since 
$$m_{V_\varphi}(y)\,dy = m_V(\varphi(y)) \,d(\varphi(y)) = m_V(\varphi(y))|\partial_z\varphi(y)|^2\,dy.$$
We define $\tilde P^{(2)}$ by the relation
$$
d(\varphi\# P^{(2)})(y,\tilde{E})=m_{V_\varphi}(y)|B_1|d\tilde{P}^{(2)}(y,\tilde{E}),
$$
and then $\tilde{P}^{(2)}$ is a $T$-invariant probability measure on $B_1\times L^p_{loc}(\R^2,\R^2)$ with first marginal $dy_{|B_1}/|B_1|$ and such that, for $\tilde{P}^{(2)}$ a.e. $(y,\tilde{E})$, $\tilde{E}\in \mathcal{A}_{m_{V_\varphi}(y)}$.\medskip

\noindent \textbf{Step 2: Application of Theorem \ref{UpBound}.} We may now  apply Theorem \ref{UpBound} to $\tilde{P}^{(1)}$ and $\tilde{P}^{(2)}$.  We thus construct a sequence $\{\nu_n^{(1)}:=\sum_{i=1}^n\delta_{x_i^{(1)}}\}$ of empirical measures on $B_1$ and a sequence $\{E_n^{(1)}\}_n$ in $L^p_{loc}(B_1,\R^2)$ such that 
$$
\divergence E_n^{(1)}=2\pi((\nu_n^{(1)})'-(m_V^{(1)})')\quad \text{ and } \quad \tilde{P}_n^1:=\dashint_{B_1}\delta_{(x,E_n^{(1)}(\sqrt{n}x+.))}dx  \to \tilde{P}^{(1)},
$$
as $n\to +\infty$. Moreover, we have
\begin{equation}\label{F1}
\limsup_{n\to +\infty} \frac{1}{n\pi} W(E_n^{(1)},\indic_{\R^2})\leq \frac{|B_1|}{\pi}\int W(E)d\tilde{P}^{(1)}(x,E).
\end{equation}
Applying now the same Theorem to $\tilde{P}^{(2)}$, we construct a sequence $\{\tilde{\nu}_n^{(2)}:=\sum_{i=1}^n\delta_{\tilde{x}_i^{(2)}}\}$ of empirical measures on $B_1$ and a sequence $\{\tilde{E}_n^{(2)}\}_n$ in $L^p_{loc}(B_1,\R^2)$ such that 
$$
\divergence \tilde{E}_n^{(2)}=2\pi((\tilde{\nu}_n^{(2)})'-(m_{V_\varphi}^{(2)})') \quad \text{ and } \tilde{P}_n^{(2)}:=\dashint_{B_1}\delta_{(x,\tilde{E}_n^{(2)}(\sqrt{n}x+.))}dx  \to \tilde{P}^{(2)},
$$
as $n\to +\infty$. Moreover, we have
\begin{equation*}
\limsup_{n\to +\infty} \frac{1}{n\pi} W(\tilde E_n^{(2)},\indic_{\R^2})\leq \frac{|B_1|}{\pi}\int W(\tilde{E})d\tilde{P}^{(2)}(y,\tilde{E}).
\end{equation*}\medskip

\noindent \textbf{Step 3: Construction of sequences and conclusion.}  It is not difficult to see that we can assume $\tilde{x}_j^{(2)}\neq 0$ for any $j$ and any $n\geq 2$ (otherwise we translate the point a little bit). Now we set $x_j^{(2)} := \vp(\tilde x_j^{(2)})$ and in view of \eqref{chgtvar}, for each $n$ we define 
$$
\nu_n^{(2)}:=\varphi\# \tilde{\nu}_n^{(2)}=\sum_{j=1}^n \delta_{x_j^{(2)}} \quad \text{ and } \quad E_n^{(2)}(x):=(D\varphi_{n^{-1/2}x})^T \tilde{E}_n^{(2)}(n^{1/2}\varphi(n^{-1/2}x)).
$$
Hence, we have a sequence of vector-fields $E_n^{(2)}$ of $L^p_{loc}(\R^2,\R^2)$ such that
$$
\divergence E_n^{(2)}=2\pi((\nu_n^{(2)})'-(m_V^{(2)})')
$$
where $m_V^{(2)}(x)=m_V(x)\indic_{\bar{B_1^c}}(x)$ is the density of $\mu_V^{(2)}$. 

Below we will use the notation $\left|\partial_z\varphi(z)\right|$ for the modulus of the complex derivative of $\varphi$ at the point $z$.

We have, for  every $i$,
\begin{align*}
&W(E_n^{(2)},\indic_{\R^2}) \\
&=\lim_{\eta\to 0} \left(\frac{1}{2}\int_{\R^2\backslash \bigcup_{i=1}^n B(x_i^{(2)},\eta)} |E_n^{(2)}(x')|^2dx' +\pi n\log\eta   \right)\\
&=\lim_{\eta\to 0} \left(\frac{1}{2}\int_{\R^2\backslash \bigcup_{i=1}^n B(x_i^{(2)},\eta)} |(D\varphi_{n^{-1/2}x'})^T \tilde{E}_n^{(2)}(n^{1/2}\varphi(n^{-1/2}x'))|^2dx' +\pi n\log\eta   \right)\\
&=\lim_{\eta\to 0} \left(\frac{1}{2}\int_{\R^2\backslash \bigcup_{i=1}^n B(y_i^{(2)},|{\partial_z\varphi}(x_i^{(2)})|\eta)} |\tilde{E}_n^{(2)}(y')|^2dy' +\pi n\log\eta   \right)\\
&=\lim_{\eta\to 0} \left(\frac{1}{2}\int_{\R^2\backslash \bigcup_{i=1}^n B(y_i^{(2)},|{\partial_z\varphi}(x_i^{(2)})|\eta)} |\tilde{E}_n^{(2)}(y')|^2dy' +\pi\sum_{i=1}^n \log|{\partial_z\varphi}(x_i)|\eta  -\pi \sum_{i=1}^n \log |{\partial_z\varphi}(x_i)| \right)\\
&=W(\tilde{E}_n^{(2)},\indic_{\R^2})-\pi \sum_{i=1}^n \log |{\partial_z\varphi}(x_i)|,
\end{align*}
where the change of variable is $y'=n^{1/2}\varphi(n^{-1/2}x')$.\\
Furthermore, we have
$$
\int W(\tilde{E})d \tilde{P}^{(2)}(y,\tilde{E})=\frac{1}{|B_1|}\int W(\tilde{E})\frac{d(\varphi \# P^{(2)})(y,\tilde{E})}{m_{V_\varphi}(y)}=\frac{1}{|B_1|}\int W\left(D\varphi_y^T E(D\varphi_y .)  \right)\frac{dP^{(2)}(x,E)}{m_{V_\varphi}(y)}
$$
by change of variable $y=\varphi(x)$ and $\tilde{E}=D\varphi_y^T E(D\varphi_y \cdot)$.\\
Now we remark that, for $\lambda>0$ and $E\in \mathcal{A}_m$,
\begin{align*}
W(\lambda E(\lambda .)) &=\lim_{R\to +\infty}\frac{1}{\pi R^2} \lim_{\eta \to 0} \left(\frac{1}{2}\int_{\R^2\backslash \bigcup_{i} B(y_i,\eta)} \chi_R(y) \lambda^2 |E(\lambda y)|^2 dy+\pi \sum_{i}\chi_R(y_i)\log\eta\right)\\
&=\lim_{R\to +\infty}\frac{1}{\pi R^2} \lim_{\eta \to 0} \left(\frac{1}{2}\int_{\R^2\backslash \bigcup_{i} B(x_i,\lambda\eta)} \chi_R(x/\lambda)|E(x)|^2dx +\pi \sum_i \chi_R(x_i/\lambda)\log\eta\right)
\end{align*}
where $x=\lambda y$. Thus, setting $R'=R\lambda$ and $\eta' =\eta \lambda$, we get
\begin{align*}
W(\lambda E(\lambda .)) 
&=\lim_{R'\to +\infty}\frac{\lambda^2}{\pi R'^2} \lim_{\eta' \to 0} \left(\frac{1}{2}\int_{\R^2\backslash \bigcup_{i} B(x_i,\eta')} \chi_{R'}(x)|E(x)|^2dx +\pi \sum_i \chi_{R'}(x_i)\left(\log\eta'-\log\eta \right)\right)\\
&=\lambda^2\left(W(E)-m\log\lambda   \right).
\end{align*}
Applying this equality with $\lambda=|{\partial_z\varphi}(x)|^{-1}=|{\partial_z\varphi^{-1}}(y)|$, we obtain
\begin{align*}
&\limsup_{n\to +\infty} \frac{1}{n\pi}\left(W(E_n^{(2)},\indic_{\R^2}) +\pi\sum_{i=1}^n\log|{\partial_z\varphi}(x_i)| \right)\\
&\leq \frac{1}{\pi}\int \frac{1}{|{\partial_z\varphi}(x)|^2}\left( W(E)+\log|{\partial_z\varphi}(x)|m_V^{(2)}(x) \right)\frac{d P^{(2)}(x,E)}{m_{V_\varphi}(y)},
\end{align*}
that is to say, because $m_V^{(2)}$ is the density of points $\{x_i\}$ as $n\to +\infty$,
\begin{align*}
&\limsup_{n\to +\infty} \frac{1}{n\pi}W(E_n^{(2)},\indic_{\R^2})+\int_{B_1^c} \log|{\partial_z\varphi}(x)|d\mu_V^{(2)}(x)\\
&\leq \frac{1}{\pi}\int W(E) \frac{d P^{(2)}(x,E)}{m_V(x)}+\int_{B_1^c} \log|{\partial_z\varphi} (x)|dP^{(2)}(x).
\end{align*}
As $\int_{B_1^c} \log|{\partial_z\varphi}(x)|d P^{(2)}(x)=\int_{B_1^c} \log|{\partial_z\varphi}(x)|d\mu_V^{(2)}(x)$, it follows that
\begin{equation}\label{F2}
\limsup_{n\to +\infty} \frac{1}{n\pi}W(E_n^{(2)},\indic_{\R^2})\leq \frac{1}{\pi}\int W(E) \frac{d P^{(2)}(x,E)}{m_V(x)}.
\end{equation}
Finally, we set
$$
\nu_{2n}:=\nu_n^{(1)}+\nu_n^{(2)} \quad \text{ and } \quad E_{2n}:=E_n^{(1)}+E_n^{(2)},
$$
and by \eqref{F1} and \eqref{F2}, we have, since $E_n^{(1)}$ and $E_n^{(2)}$ have disjoint supports, 
\begin{align*}
\limsup_{n\to +\infty} \frac{1}{n\pi} W(E_n,\indic_{\R^2})&\leq \frac{1}{\pi}\int \frac{W(E)}{m_V(x)}dP^{(1)}(x,E)+\frac{1}{\pi}\int \frac{W(E)}{m_V(x)}dP^{(2)}(x,E)\\
&=\frac{1}{\pi}\int \frac{W(E)}{m_V(x)}dP(x,E)
\end{align*}
which proves \eqref{UB}. Furthermore, by changes of variable, 
$$
P_n^{(1)}:=\int_{B_1}\delta_{(x,E_n^{(1)}(x\sqrt{n}+.))}d\mu_V(x)\to P^{(1)} \quad \text{and} \quad P_n^{(2)}:=\int_{B_1^c}\delta_{(x,E_n^{(2)}(x\sqrt{n}+.))}d\mu_V(x)\to P^{(2)}
$$
in the weak sense of measure, and it follows that
$$
P_n=P_n^{(1)}+P_n^{(2)}\to P^{(1)}+P^{(2)}= P.\\
$$
Part \textbf{C} follows from \textbf{A} and \textbf{B} as in \cite{2DSandier}.

\section{Consequence: the Logarithmic Energy on the Sphere}\label{CTLEOTS}
The asymptotic expansion of the minimum of the Hamiltonian $w_n$ in the case of weakly confining potential that we have --- where the minimizing points are allowed to fill the whole plane instead of being confined to a fixed compact set as in the classical case --- allows through the use of the inverse stereographic projection  (as in \cite{Electsphere}, \cite{Dragnev:2002}, \cite{Hardy})   to determine the asymptotic expansion of the optimal logarithmic energy on sphere.

\subsection{Inverse stereographic projection}
Here we recall properties of the inverse stereographic projection used by Hardy and Kuijlaars \cite{Hardy,Hardy2} and by Bloom, Levenberg and Wielonsky \cite{Bloom:2014qy}  in order to prove Theorem \ref{equi}.\\
Let $\mathcal{S}$ be the sphere of $\R^3$ centred at $(0,0,1/2)$ of radius $1/2$, $\Sigma$ be an unbounded closed set of $\R^2$ and $T:\R^2\to \mathcal{S}$ be the associated inverse stereographic projection defined by
\begin{equation*}
T(x_1,x_2)=\left(\frac{x_1}{1+|x|^2}, \frac{x_2}{1+|x|^2},\frac{|x|^2}{1+|x|^2} \right), \quad \text{for any $x=(x_1,x_2)\in\R^2$},
\end{equation*}
where   $\R^2:=\{(x_1,x_2,0);x_1,x_2\in\R\}$. We know that $T$ is a conformal homeomorphism from $\C$ to $\mathcal{S}\backslash \{N\}$ where $N :=(0,0,1)$ is the North pole of $\mathcal{S}$. \\

\noindent We have the following identity:
\begin{equation*}\label{distanceT}
\|T(x)-T(y)\|=\frac{|x-y|}{\sqrt{1+|x|^2}\sqrt{1+|y|^2}}, \text{ for any } x,y\in \R^2.
\end{equation*}
Furthermore, if $|y|\to +\infty$, we obtain, for any $x\in\R^2$:
\begin{equation} \label{interinf} \|T(x)-N\|=\frac{1}{\sqrt{1+|x|^2}}.
\end{equation}

\noindent We note $\Sigma_{\mathcal{S}}=T(\Sigma)\cup \{N\}$ the closure of $T(\Sigma)$ in $\mathcal{S}$. Let $\mathcal{M}_1(\Sigma)$ be the set of probability measures on $\Sigma$. For $\mu \in  \mathcal{M}_1(\Sigma)$, we denote by $T\#\mu$ its push-forward measure by $T$ characterized by
\begin{equation*}
\int_{\Sigma_{\mathcal{S}}}f(z)dT\#\mu(z)=\int_\Sigma f(T(x))d\mu(x),
\end{equation*}
for every Borel function $f:\Sigma_{\mathcal{S}}\to \R$. The following result is proved in \cite{Hardy}:
\begin{lemma} \label{THomeo} The correspondance  $\mu\to T\#\mu$ is a homeomorphism from the space $\mathcal{M}_1(\Sigma)$ to the set of $ \mu\in \mathcal{M}_1(\Sigma_S)$ such that $ \mu(\{N\})=0$.
\end{lemma}

\subsection{Asymptotic expansion of the optimal logarithmic energy on the unit sphere}
An important case is the equilibrium measure associated to the potential 
$$
V(x)=\log(1+|x|^2)
$$
 corresponding to the external field $\mathcal{V}\equiv 0$ on $\mathcal{S}$ and where $T\#\mu_V$ is the uniform probability measure on $\mathcal{S}$ (see \cite{Hardy}). Hence $V$ is an admissible potential and from \eqref{mv} we find 
$$
\displaystyle d\mu_V(x)=\frac{dx}{\pi(1+|x|^2)^2}\quad \text{and} \quad \Sigma_V=\R^2.
$$

We define
$$
 \overline{w}_n(x_1,...,x_n):=-\sum_{i\neq j}^n\log|x_i-x_j|+(n-1)\sum_{i=1}^n\log(1+|x_i|^2),
 $$
and we recall that the logarithmic energy of a configuration $(y_1,...,y_n)\in \mathcal{S}^n$ is given by
$$
\textnormal{E}_{\log}(y_1,....,y_n):=-\sum_{i\neq j}^n\log\|y_i-y_j\|.
$$
Furthermore, we recall that $\mathcal{E}_{\log}(n)$ denotes the minimal logarithmic energy of $n$ points on $\S^2$.
\begin{lemma} \label{wminimizer}
For any $(x_1,...,x_n)\in (\R^2)^n$, we have the following equalities:
\begin{equation*}
\overline{w}_n(x_1,...,x_n)=\textnormal{E}_{\log}(T(x_1),...,T(x_n)) \quad \text{ and } \quad w_n(x_1,...,x_n)=\textnormal{E}_{\log}(T(x_1),...,T(x_n), N),
\end{equation*}
which imply that
\begin{align*}
&(x_1,...,x_n) \text{ minimizes } \overline{w}_n \iff (T(x_1),...,T(x_n)) \text{ minimizes } \textnormal{E}_{\log}\\
&(x_1,...,x_n) \text{ minimizes } w_n \iff (T(x_1),...,T(x_n),N) \text{ minimizes } \textnormal{E}_{\log}.
\end{align*}
\end{lemma}
\begin{proof}
For any $1\leq i \leq n$, we set $y_i:=T(x_i)$, hence we get, by \eqref{distanceT},
\begin{align*}
\textnormal{E}_{\log}(y_1,....,y_n)&:=-\sum_{i\neq j}^n\log \|y_i-y_j\|\\
&=-\sum_{i\neq j}^n\log\|T(x_i)-T(x_j)\| \\
&=-\sum_{i\neq j}^n\log\left(  \frac{|x_i-x_j|}{\sqrt{1+|x_i|^2}\sqrt{1+|x_j|^2}} \right) \\
&=-\sum_{i\neq j}^n\log|x_i-x_j|+(n-1)\sum_{i=1}^n\log(1+|x_i|^2)\\
&=\overline{w}_n(x_1,...,x_n).
\end{align*}
Furthermore, by \eqref{interinf}, we obtain
\begin{align*}
w_n(x_1,...,x_n)&=\overline{w}_n(x_1,...,x_n)+\sum_{i=1}^n\log(1+|x_i|^2)\\
&=-\sum_{i\neq j}\log\|y_i-y_j\|-2\sum_{i=1}^n\log\|y_i-N\|=\textnormal{E}_{\log}(y_1,....,y_n,N).
\end{align*}
\end{proof}
\begin{lemma} \label{weakmeas}
If $(x_1,...,x_n)$ minimizes $w_n$ or $\overline{w}_n$, then, for $\displaystyle \nu_n:=\frac1n\sum_{i=1}^n\delta_{x_i}$, we have
 $$
\nu_n\to \mu_V, \quad \text{ as } n\to +\infty,
 $$
 in the weak sense of measures.
\end{lemma}
\begin{proof}
Let $(x_1,...,x_n)$ be a minimizer of $\bar{w}_n$, then $(T(x_1),...,T(x_n))$ is a minimizer of $\textnormal{E}_{\log}$.
Brauchart, Dragnev and Saff proved in \cite[Proposition 11]{EquilSphere} that
$$
\frac{1}{n}\sum_{i=1}^n\delta_{T(x_i)}\to T\#\mu_V.
$$
As $T\#\mu_V(\{N\})=0$, by Lemma \ref{THomeo} we get the result.\\
If $(x_1,...,x_n)$ is a minimizer of $w_n$, then $(T(x_1),...,T(x_n), N)$ minimizes $\textnormal{E}_{\log}$ and we can use our previous argument because
$$
\frac{1}{n+1}\left(\sum_{i=1}^n \delta_{T(x_i)}+\delta_N\right)=\frac{1}{n}\sum_{i=1}^n\delta_{T(x_i)}\left( \frac{n}{n+1} \right)+\frac{\delta_N}{n+1}\to T\# \mu_V,
$$
in the weak sense of measures, and we have the same conclusion.
\end{proof}

\begin{lemma}\label{CVlog} If $(x_1,...,x_n)$ is a minimizer of $w_n$   and if $\displaystyle \nu_n:=\frac1n\sum_{i=1}^n\delta_{x_i}$ then
$$
\lim_{n\to +\infty}\int_{\R^2}\log(1+|x|^2)d\nu_n(x)=\int_{\R^2}\log(1+|x|^2)d\mu_V(x).
$$
There exists minimizers of $\overline{w}_n$ for which the same is true.
\end{lemma}
\begin{proof} Let $(x_1,...,x_n)$ be a minimizer of $\overline{w}_n$. We define $y_i:=T(x_i)$ for any $1\leq i\leq n$ and we notice that
$$\int_{\R^2}\log(1+|x|^2)d\nu_n(x)=-2\int_{\R^2}\log\left(\frac{1}{\sqrt{1+|x|^2}}  \right)d\nu_n(x)=-2\int_{\mathcal{S}}\log\|y-N\|dT\#\nu_n(y),
$$
and by Lemma~\ref{wminimizer}, $(y_1,...,y_n)$ is a minimizer of $\textnormal{E}_{\log}$ on $\mathcal{S}$. 

Now, denoting by $\sigma$ the normalized Haar measure on $\text{SO(3)}$, for any point $y_0$ on the sphere we have that the image of $\sigma$ by the map $R\to R(y_0)$ is the normalized uniform measure on the sphere. Therefore 
$$-\int_{\text{SO(3)}}\left(\frac{2}{n}\sum_i \log\|Ry_i-N\|\right)\,d\sigma(R) = -2\dashint_{\ms}\log\|y-N\|\,dy,$$
where $\dashint$ denotes the average with respect to the uniform measure on $\ms$. It follows that for some $R_1$ the integrand of left-hand side is no greater than the right-hand side and that for some (possibly identical) $R_2$ the reverse is true. Then since $\text{SO(3)}$ is connected we may connect $R_1$ and $R_2$ by a continuous path, and we may further assume that $R y_i \neq N$ for every $i$ when $R$ is along this path. Since  the integrand of the left-hand side  is continuous with respect to $R$ on the path we deduce that there exists a rotated configuration $(\bar y_1,...,\bar y_n)$ such that 
$$\frac{1}{n}\sum_i \log\|\bar y_i-N\| = \dashint_{\ms}\log\|y-N\|\,dy.$$
But, for any rotation $R$ of $\ms$ the rotated configuration of points is still a minimizer. Thus, transporting   back to $\R^2$ with $T^{-1}$, we obtain a minimizer $(\bar x_1,...,\bar x_n)$ of $\overline{w}_n$ such that 
$$\frac{1}{n}\sum_i\log(1+|\bar x_i|^2) =\int_{\R^2}\log(1+|x|^2)d\mu_V(x).$$

If $(x_1,\dots,x_n)$ is a minimizer of $w_n$ we use \cite[Theorem 15]{EquilSphere} about the optimal point separation which yields the existence of constants $C$ and $n_0$ such that for any $n\geq n_0$ and any minimizer $\{y_1,...,y_n\}\in\mathcal{S}^n$ of the logarithmic energy on the sphere, we have 
$$
\min_{i\neq j}\|y_i-y_j\|>\frac{C}{\sqrt{n-1}}.
$$
Letting $y_i = T(x_i)$ we have that $(N,y_1,\dots, y_n)$ is a minimizer of the logarithmic energy, hence for any $1\leq i\leq n$,
\begin{equation*} 
\|y_i-N\|>\frac{C}{\sqrt{n-1}}.
\end{equation*}

\noindent For $n\geq n_0$ and $\delta>0$ sufficiently small, we define, for any $0< r\leq \delta$, 
$$
k(r):=\string #\left\{y_i \mid y_i\in B(N,r)\cap\mathcal{S}  \right\},
$$
and $r_i=\|y_i-N\|$. From the separation property  there exists a constant $C$ such that $k(r)\leq Cr^2n$ for any $r$. Hence we have $k(r) = 0$ if $r<1/\sqrt{Cn}$. Thus, using integration by parts, for some small enough $c>0$ we have
\begin{align*}
-\sum_{y_i\in B(N,\delta)}\log r_i &=-\int_{c/\sqrt{n-1}}^\delta \log r\,dk(r)\\
&=-n(\delta)\log \delta +\int_{c/\sqrt{n-1}}^\delta \frac{k(r)}{r}dr\\
&\leq -Cn \delta^2\log\delta +Cn\int_{c/\sqrt{n-1}}^\delta rdr\leq C\delta^2 n|\log\delta|.
\end{align*}
It follows that
\begin{align}\label{convlog}
\lim_{\delta\to 0}\limsup_{n\to +\infty}-\frac{1}{n}\int_{B(N,\delta)\cap\mathcal{S}}\log\|y-N\|dT\#\nu_n(y)&=   \lim_{\delta\to 0}\limsup_{n\to +\infty}-\frac{1}{n}\sum_{y_i\in B(N,\delta)}\log\|y_i-N\|=0.
\end{align}
For every integer $n$ and $R>0$ we have 
\begin{equation}\label{aaa}  \int_{\R^2}\log(1+|x|^2)d\nu_n(x) = \int_{B_R}\log(1+|x|^2)d\nu_n(x)+\int_{B_R^c}\log(1+|x|^2)d\nu_n(x).\end{equation}
By Lemma \ref{weakmeas}, $\nu_n$ goes weakly to the measure $\mu_V$ on $B_R$ for any $R$  hence
$$\lim_{R\to +\infty}\lim_{n\to +\infty}\int_{B_R}\log(1+|x|^2)d\nu_n(x) =\lim_{R\to +\infty}\int_{B_R}\log(1+|x|^2)d\mu_V(x) = \int_{\R^2}\log(1+|x|^2)d\mu_V(x),$$
and from \eqref{convlog} we have     
$$\lim_{R\to +\infty}\limsup_{n\to +\infty}\frac{1}{n}\int_{B_R^c}\log(1+|x|^2)d\nu_n(x) = 0.$$
Therefore, taking the limits $n\to +\infty$ and then $R\to +\infty$ in \eqref{aaa} we find 
$$lim_{n\to +\infty}\int_{\R^2}\log(1+|x|^2)d\nu_n(x) =\int_{\R^2}\log(1+|x|^2)d\mu_V(x).$$
The convergence is proved.
\end{proof}

\noindent The following result proves the existence of the constant $C$ in the Conjecture \ref{conj1} of Rakhmanov, Saff and Zhou.
\begin{thm}\label{RSZ} We have
\begin{equation*}
\mathcal{E}_{\log}(n)=\left( \frac{1}{2}-\log 2 \right)n^2-\frac{n}{2}\log n + \left( \frac{1}{\pi}\min_{\mathcal{A}_1}W +\frac{\log\pi}{2}+\log2\right)n +o(n), \quad \text{as } n\to +\infty.
\end{equation*}
\end{thm}
\begin{proof}
As $\textnormal{E}_{\log}$ is invariant by translation of the $2$-sphere, we work on the sphere $\tilde{\S}^2$ of radius 1 and centered at $(0,0,1)$. Let $(y_1,...,y_n)\in\tilde{\S}^2$ be a minimizer of $\textnormal{E}_{\log}$. Without loss of generality, for any $n$, we can choose this configuration such that $y_i\neq N$ for any $1\leq i\leq n$. Hence there exists $(x_1,...,x_n)$ such that $\displaystyle \frac{y_i}{2}=T(x_i)$ for any $i$ and we get

\begin{align*} 
\textnormal{E}_{\log}(y_1,....,y_n)&=-\sum_{i\neq j}^n\log \|y_i-y_j\|\\
&=-\sum_{i\neq j}^n\log\|T(x_i)-T(x_j)\|-n(n-1)\log 2 \\
&=\overline{w}_n(x_1,...,x_n)-n(n-1)\log 2.
\end{align*}
By Lemma \ref{wminimizer}, $(y_1,...,y_n)$ is a minimizer of $\textnormal{E}_{\log}$ if and only if $(x_1,...,x_n)$ is a minimizer of $\overline{w}_n$. By the lower bound \eqref{LB} and the convergence of Lemma \ref{CVlog}, we have, for some  minimizer $(\bar{x}_1,...,\bar{x}_n)$ of $\overline{w}_n$:
\begin{align*}
&\liminf_{n\to +\infty}\frac{1}{n}\left[ \overline{w}_n(\bar{x}_1,...,\bar{x}_n)-n^2 I_V(\mu_V)+\frac{n}{2}\log n  \right]\\
&=\liminf_{n\to +\infty}\frac{1}{n}\left[ w_n(\bar{x}_1,...,\bar{x}_n)-\sum_{i=1}^n\log(1+|\bar{x}_i|^2)-n^2 I_V(\mu_V)+\frac{n}{2}\log n  \right]\\
&\geq \alpha_V -\int_{\R^2}\log(1+|x|^2)d\mu_V(x).
\end{align*}
The upper bound \eqref{UB} and Lemma \ref{convlog} yield, $(x_1,...,x_n)$ being a minimizer of $w_n$:
\begin{align*}
&\limsup_{n\to +\infty} \frac{1}{n}\left[ \overline{w}_n(\bar{x}_1,...,\bar{x}_n)-n^2 I_V(\mu_V)+\frac{n}{2}\log n  \right]\\
&\leq \limsup_{n\to +\infty} \frac{1}{n}\left[ \overline{w}_n(x_1,...,x_n)-n^2 I_V(\mu_V)+\frac{n}{2}\log n  \right]\\
&= \limsup_{n\to +\infty}\frac{1}{n}\left[ w_n(x_1,...,x_n)-\sum_{i=1}^n\log(1+|x_i|^2)-n^2 I_V(\mu_V)+\frac{n}{2}\log n  \right]\\
&=\alpha_V -\int_{\R^2}\log(1+|x|^2)d\mu_V(x).
\end{align*}
Thus, we get
$$
\lim_{n\to +\infty}\frac{1}{n}\left[ \overline{w}_n(\bar{x}_1,...,\bar{x}_n)-n^2 I_V(\mu_V)+\frac{n}{2}\log n  \right]=\alpha_V -\int_{\R^2}\log(1+|x|^2)d\mu_V(x).
$$
Therefore, we have the following asymptotic expansion, as $n\to +\infty$, for some minimizer $(\bar{x}_1,...,\bar{x}_n)$ of $\overline{w}_n$:
\begin{align*}
&\overline{w}_n(\bar{x}_1,...,\bar{x}_n)\\
&=n^2I_V(\mu_V)-\frac{n}{2}\log n+\left(\frac{1}{\pi}\min_{\mathcal{A}_1}W-\frac{1}{2}\int_{\R^2}m_V(x)\log m_V(x)dx-\int_{\R^2}V(x)d\mu_V(x)\right)n + o(n). 
\end{align*}
We know that $\displaystyle I_V(\mu_V)=\frac{1}{2}$ (see \cite[Eq. (2.26)]{BrauchartSphere}) and 
\begin{align*}
\int_{\R^2}\log(1+|x|^2)d\mu_V(x)&=\frac{1}{\pi}\int_{\R^2}\frac{\log(1+|x|^2)}{(1+|x|^2)^2}\,dx\\
&=2\int_0^{+\infty}\frac{r\log(1+r^2)}{(1+r^2)^2}dr\\
&=-\left[\frac{\log(1+r^2)}{1+r^2}  \right]_0^{+\infty}+\int_0^{+\infty}\frac{2r}{(1+r^2)^2}dr\\
&=-\left[ \frac{1}{1+r^2} \right]_0^{+\infty}\\
&=1.
\end{align*}
Hence we obtain, as $n\to +\infty$, 
\begin{align*}
\overline{w}_n(\bar{x}_1,...,\bar{x}_n)&=\frac{n^2}{2}-\frac{n}{2}\log n + \left(   \frac{1}{\pi}\min_{\mathcal{A}_1}W +\frac{1}{2}\int \log(\pi(1+|x|^2)^2)d\mu_V(x)-1\right)n + o(n)\\
&=\frac{n^2}{2}-\frac{n}{2}\log n +\left(   \frac{1}{\pi}\min_{\mathcal{A}_1}W +\frac{\log\pi}{2}+\int \log(1+|x|^2)d\mu_V(x) -1\right)n +o(n) \\
&= \frac{n^2}{2}-\frac{n}{2}\log n +\left(   \frac{1}{\pi}\min_{\mathcal{A}_1}W +\frac{\log\pi}{2}\right)n +o(n),
\end{align*}
and the asymptotic expansion of $\mathcal{E}_{\log}(n)$ is
$$
\mathcal{E}_{\log}(n)=\left( \frac{1}{2}-\log 2 \right)n^2-\frac{n}{2}\log n + \left( \frac{1}{\pi}\min_{\mathcal{A}_1}W +\frac{\log\pi}{2}+\log 2\right)n +o(n).
$$
\end{proof}
\begin{remark}
It follows from the lower bound proved by Rakhmanov, Saff and Zhou \cite[Theorem 3.1]{2400418}, that
\begin{align*}
 \frac{1}{\pi}\min_{\mathcal{A}_1}W +\frac{\log\pi}{2}+\log 2&=\lim_{n\to +\infty}\frac{1}{n}\left[\textnormal{E}_{\log}(y_1,...,y_n)-\left(\frac{1}{2}-\log2\right)n^2+\frac{n}{2}\log n  \right]\\
 &\geq -\frac{1}{2}\log\left[\frac{\pi}{2}(1-e^{-a})^b  \right],
\end{align*}
where $\displaystyle a:=\frac{2\sqrt{2\pi}}{\sqrt{27}}\left(\sqrt{2\pi+\sqrt{27}} +\sqrt{2\pi} \right)$ and $\displaystyle b:=\frac{\sqrt{2\pi+\sqrt{27}}-\sqrt{2\pi}}{\sqrt{2\pi+\sqrt{27}}+\sqrt{2\pi}}$, and we get
$$
\min_{\mathcal{A}_1}W\geq -\frac{\pi}{2}\log\left[ 2\pi^2(1-e^{-a})^b \right]\approx - 4.6842707.
$$
\end{remark}

\subsection{Computation of renormalized energy for the triangular lattice and upper bound for the term of order $n$}\label{NTh}
Sandier and Serfaty proved in \cite[Lemma 3.3]{Sandier_Serfaty} that
\begin{equation*}
W(\Lambda_{1/2\pi})=-\frac{1}{2}\log\left(\sqrt{2\pi b}|\eta(\tau)|^2\right),
\end{equation*}
where $\Lambda_{1/2\pi}$ is the triangular lattice corresponding to the density $m=1/2\pi$, $\displaystyle \tau=a+ib=1/2+i\frac{\sqrt{3}}{2}$ and $\eta$ is the Dedekind eta function defined, with $q=e^{2i\pi\tau}$, by
\begin{equation*}
\eta(\tau)=q^{1/24}\prod_{n\geq 1}(1-q^n).
\end{equation*}
We recall the Chowla-Selberg formula (see \cite{SelbergChowla} or \cite[Proposition 10.5.11]{Cohen} for details):
\begin{equation*}
4\pi\sqrt{-D}b|\eta(\tau)|^4 = \prod_{m=1}^{|D|}\Gamma\left(\frac{m}{|D|}  \right)^{\frac{w}{2}\left( \frac{D}{m} \right)},
\end{equation*}
for $\tau$ a root of the integral quadratic equation $\alpha z^2+\beta z+\gamma=0$ where $D=\beta^2-4\alpha\gamma<0$, $\displaystyle \left(\frac{D}{m}  \right)$ is the Kronecker symbol, $w$ the number of roots of unity in $\Q(i\sqrt{-D})$ and when the class number of $\Q(i\sqrt{-D})$ is equal to 1. In our case $b=\sqrt{3}/2$, $w=6$, $\alpha=\beta=\gamma=1$ because $\tau$ is a root of unity, hence $D=-3$, $\displaystyle\left(  \frac{-3}{1}\right)=1$ and $\displaystyle \left(  \frac{-3}{2}\right)=-1$ by the Gauss Lemma. 
%Therefore we obtain
%\begin{equation*}
%\Gamma(1/3)^3\Gamma(2/3)^{-3}=4\pi\sqrt{3}\frac{\sqrt{3}}{2}|\eta(\tau)|^4,
%\end{equation*}
%and by Euler's reflection formula $\displaystyle \Gamma(1-1/3)\Gamma(1/3)=\frac{\pi}{\sin(\pi/3)}$, we get
%\begin{equation*}
%\frac{\Gamma(1/3)^6 3\sqrt{3}}{8\pi^3}=\frac{4\pi\sqrt{3}\times\sqrt{3}|\eta(\tau)|^4}{2}.
%\end{equation*}
Finally we obtain
\begin{equation*}|\eta(\tau)|^4=\frac{\Gamma(1/3)^6\sqrt{3}}{16\pi^4}.
\end{equation*}
Now it is possible to find the exact value of the renormalized energy of the triangular lattice $\Lambda_1$ of density $m=1$:
\begin{align*}
W(\Lambda_1)&=2\pi W(\Lambda_{1/2\pi})-\pi\frac{\log(2\pi)}{2}\\
&=-\pi \log\left(\sqrt{2\pi b}|\eta(\tau)|^2\right)-\pi\frac{\log(2\pi)}{2}\\
&= \pi\log\pi -\frac{\pi}{2}\log3-3\pi\log(\Gamma(1/3))+\frac{3}{2}\pi\log2\\
&=\pi\log\left( \frac{2\sqrt{2}\pi}{\sqrt{3}\Gamma(1/3)^3} \right)\\
&\approx  -4.1504128.
\end{align*}
Thus, we get
\begin{align*}
&\frac{1}{\pi}W(\Lambda_1)+\frac{\log\pi}{2}+\log 2\\
&=\frac{1}{\pi}\left( \pi\log\pi -\frac{\pi}{2}\log3-3\pi\log(\Gamma(1/3))+\frac{3}{2}\pi\log2  \right)+\frac{\log\pi}{2}+\log 2\\
&=2\log2 +\frac{1}{2}\log\frac{2}{3}+3\log\frac{\sqrt{\pi}}{\Gamma(1/3)}=C_{BHS}\approx  -0.0556053,
\end{align*}
and we find exactly the value $C_{BHS}$ conjectured by Brauchart, Hardin and Saff in \cite[Conjecture 4]{Brauchart}. Therefore Conjecture \ref{conj2} is true if and only if the triangular lattice $\Lambda_1$ is a global minimizer of $W$ among vector-fields in $\mathcal{A}_1$, i.e.
$$
\min_{\mathcal{A}_1}W=W(\Lambda_1)=\pi\log\left( \frac{2\sqrt{2}\pi}{\sqrt{3}\Gamma(1/3)^3} \right).
$$
Thus we obtain the following result
\begin{thm}\label{equiconj} We have:
\begin{enumerate}
\item  It holds 
\begin{equation*}
\lim_{n\to +\infty}\frac{1}{n}\left[\mathcal{E}_{\log}(n)-\left(\frac{1}{2}-\log2\right)n^2+\frac{n}{2}\log n  \right]\leq 2\log2 +\frac{1}{2}\log\frac{2}{3}+3\log\frac{\sqrt{\pi}}{\Gamma(1/3)}.
\end{equation*}
\item Conjectures \ref{conj2} and \ref{conj3} are equivalent, i.e. $\displaystyle \min_{\mathcal{A}_1} W= W(\Lambda_1) \iff C=C_{BHS}$.
\end{enumerate}
\end{thm}

\section*{Appendix} Here we prove the following  
\begin{prop} Assume $X$ is  a Polish space X,   on which $\R^n$ acts continuously. We denote this action $(\lambda,u)\to\theta_\lambda u$ and assume it is separately continuous w.r.t  both $\lambda\in\R^n$ and $u\in X$. Assume  $P$ is a probability measure on $\R^n\times X$ which for every $\lambda$ is invariant under the map $(x,u)\to (x,\theta_\lambda u)$. Then, for any continuous function $x\to\lambda(x)$ it holds that $P$ is invariant under the map $(x,u)\to (x,\theta_\lambda(x) u)$.
\end{prop}
\begin{proof}
Let $\Phi$ be any bounded continuous function on $\R^n\times X$, we need to prove that for any continuous function $x\to\lambda(x)$
\begin{equation*}
\int \Phi(x,u)\,dP(x,u) = \int \Phi(x,\theta_{\lambda(x)}u)\,dP(x,u).
\end{equation*}
for any integer $k>0$ we let $\{\chi_{i,k}\}_i$ be a partition of unity on $\R^n$ subordinate to the covering of $\R^n$ by balls of radius $1/k$, and we let $x_{i,k}$ belong to the support of $\chi_{i,k}$. Then, from the continuity of $\Phi$, $\lambda$ and $\theta$, it is straightforward to check that for every $(x,u)\in\R^n\times X$ we have 
$$\lim_{k\to +\infty} \sum_i \chi_{i,k}(x) \Phi(x,\theta_{\lambda(x_{i,k})}u) = \Phi(x,\theta_{\lambda(x)}u).$$ 
It follows by dominated convergence that 
\begin{equation}\label{jnh}\lim_{k\to +\infty} \sum_i\int \chi_{i,k}(x) \Phi(x,\theta_{\lambda(x_{i,k})}u) \,dP(x,u) = \int \Phi(x,\theta_{\lambda(x)}u)\,dP(x,u).\end{equation}
But by the invariance of $P$ we have 
$$ \int \chi_{i,k}(x) \Phi(x,\theta_{\lambda(x_{i,k})}u) \,dP(x,u) = \int \chi_{i,k}(x) \Phi(x,u) \,dP(x,u),$$
hence 
$$ \sum_i\int \chi_{i,k}(x) \Phi(x,\theta_{\lambda(x_{i,k})}u) \,dP(x,u) = \int \Phi(x,u) \,dP(x,u).$$
Replacing \eqref{jnh} we get the desired result.
\end{proof}

\noindent \textbf{Acknowledgements:} We are grateful to Adrien Hardy, Edward B. Saff and Sylvia Serfaty for their interest and helpful discussions. We are also grateful to the  anonymous referees for their suggestions, remarks and patience in reading the manuscript.

\bibliographystyle{plain}
\bibliography{allbiblio}

\noindent LAURENT BETERMIN \\
Institut f\"{u}r Angewandte Mathematik, \\
Interdisciplinary Center for Scientific Computing (IWR),\\
Universit\"{a}t Heidelberg,\\
Im Neuenheimer Feld 205, 69120 Heidelberg. Deutschland \\
\texttt{betermin@uni-heidelberg.de}\\ 

\noindent ETIENNE SANDIER\\
Universit\'e Paris-Est,\\
LAMA - CNRS UMR 8050,\\
61, Avenue du G\'en\'eral de Gaulle, 94010 Cr\'eteil. France\\
\& Institut Universitaire de France\\
\texttt{sandier@u-pec.fr}

\end{document}